\documentclass[12pt,leqno]{article}
\usepackage{amsmath, amsthm, amsfonts, amssymb, color}
\usepackage{mathrsfs}
\usepackage[notcite,notref]{showkeys}

\newtheorem{thm}{Theorem}[section]
\newtheorem{cor}[thm]{Corollary}
\newtheorem{lem}[thm]{Lemma}

\newtheorem{exa}[thm]{Example}
\newtheorem{Remark}[thm]{Remark}
\newtheorem{defn}[thm]{Definition}
\newcommand{\scr}[1]{\mathscr #1}
\definecolor{wco}{rgb}{0.5,0.2,0.3}

\numberwithin{equation}{section}

\newcommand{\ua}{\uparrow}

\title{{\bf Quasi-invariant theorem on the Gaussian path space}\footnote{Supported in
		part by  NNSFC (12071085).} }
\author{
	{\bf Qinpin Chen, Jian Sun, Bo Wu}\\
	\footnotesize { School  of Mathematical Sciences, Fudan
		University, Shanghai 200433, China}
}
\date{\today}
\begin{document}
	\maketitle
	
	\def\R{\mathbb R}
	\def\EE{\mathbb E}
	\def\Z{\mathbb Z}
	\def\ff{\frac}
	\def\ss{\sqrt}
	\def\H{\mathbb H}
	\def\dd{\delta}
	\def\DD{\Delta}
	\def\vv{\varepsilon}
	\def\rr{\rho}
	\def\<{\langle}
	\def\>{\rangle}
	\def\GG{\Gamma}
	\def\gg{\gamma}
	\def\ll{\lambda}
	\def\LL{\Lambda}
	\def\nn{\nabla}
	\def\pp{\partial}
	\def\d{\text{\rm{d}}}
	\def\Id{\text{\rm{Id}}}
	\def\loc{\text{\rm{loc}}}
	\def\bb{\beta}
	\def\aa{\alpha}
	\def\D{\scr D}
	\def\E{\scr E}
	\def\si{\sigma}
	\def\ess{\text{\rm{ess}}}
	\def\beg{\begin}
	\def\beq{\beg}
	\def\F{\scr F}
	\def\Ric{\text{\rm{Ric}}}
	\def\Var{\text{\rm{Var}}}
	\def\Ent{\text{\rm{Ent}}}
	\def\Hess{\text{\rm{Hess}}}
	\def\B{\scr B}
	\def\e{\text{\rm{e}}}
	\def\ua{\underline a}
	\def\OO{\Omega}
	\def\b{\mathbf b}
	\def\oo{\omega}
	\def\tt{\tilde}
	\def\Ric{\text{\rm{Ric}}}
	\def\cut{\text{\rm{cut}}}
	\def\P{\mathbb P}
	\def\ifn{I_n(f^{\bigotimes n})}
	\def\fff{f(x_1)\dots f(x_n)}
	\def\ifm{I_m(g^{\bigotimes m})}
	\def\ee{\varepsilon}
	\def\C{\scr C}
	\def\M{\scr M}
	\def\ll{\lambda}
	\def\X{\scr X}
	\def\T{\scr T}
	\def\A{\mathbf A}
	\def\LL{\scr L}
	\def\LLL{\Lambda}
	\def\gap{\mathbf{gap}}
	\def\div{\text{\rm div}}
	\def\Lip{\text{\rm Lip}}
	\def\dist{\text{\rm dist}}
	\def\cut{\text{\rm cut}}
	\def\supp{\text{\rm supp}}
	\def\Cov{\text{\rm Cov}}
	\def\Dom{\text{\rm Dom}}
	\def\Cap{\text{\rm Cap}}
	\def\II{{\mathbb I}}
	\def\beq{\beg{equation}}
	\def\sect{\text{\rm sect}}
	\def\H{\mathbb H}
	\def\N{\mathbb N}
	\def\K{\scr K}
	\def\RR{\scr R}
	
	\begin{abstract}In this article, we will first introduce a class of Gaussian processes, and prove the quasi-invariant theorem with respect to the Gaussian Wiener measure, which is the law of the associated Gaussian process. In particular, it includes the case of the fractional Brownian motion.

	As applications, we will establish the integration by parts formula and Bismut-Elworthy-Li formula on the Gaussian path space, and by which we derive some logarithmic Sobolev inequalities. Moreover, we will also provides some applications in the field of financial mathematics.
	\end{abstract}

	\vskip 2cm
	
	\section{Introduction}
	Since Cameron and Martin \cite{CM1} established the classical quasi-invariant theorem with respect to the standard Wiener measure on $C([0,1];\mathbb{R}^n)$, stochastic analysis on the infinite dimensional space had been well developed and widely used in other directions. In the field of mathematical research, there are two very important aspects: One is the analysis on the Riemannian path and loop spaces, see \cite{CM1, D1, FM, F3, Hsu1} and references therein; the other is the analysis on the fractional path spaces, see \cite{K, MV, HDJ, N, JM} and references therein.

	In  mathematical finance area, people usually uses Brownian motion to model asset prices. This approach achived hugh successes in derivatives pricing which cultivated in 1997 Nobel Economcs prize to the works of Black, Scholes and Merton who developed a complete theory of option pricing formula. The theory is complete but still has a serious unsolved discrepancies between the constant volatility used in the model and realistic market data. In the Black-Scholes-Merton theory underlying price has a constant volatility which is independent of the option strikes but market option prices implies a volatility curve with respect to strikes. This phenomina called for more models afterwards among which stochastic volatility model played a key roles in applications, for example \cite{Heston}. One direction is to deviate from the traditional Brownian motion to fractional Brownian motion as a main driver for the volatility process \cite{AL}.

	Fractional Brownian motion (fBm for short) is a special Gaussian process and also been regarded as an appropriate tool of mathematical finance under the complex system science system, and it has important theoretical and practical significance as a generalization and deepening of the Wiener process that describes the price fluctuation behavior model of financial markets \cite{EV, MV}. In order to more accurately characterize the complex financial behavior of financial market price fluctuations, it is necessary to introduce a more general Gaussian process.

	In this paper, we will first introduce a class of new $F$-Gaussian processes in Section $2$, and the integral for F-Gaussian process will be presented in Section $3$. In Section $4$ we will prove quasi-invariant theorem with respect to the associated Gaussian Wiener measure, and by which the integration by parts formula will established. In Section $5$, three class of Logarithmic Sobolev inequalities will be obtained. Moreover, we will prove Bismut's formula and Martingale for F-Gaussian process in Section $6$ and Section $7$ respectively. Finally, we will also provides some applications in the field of financial mathematics in Section $8$.

	\section{$F$-Gaussian process}
	
	Let $(\Omega,\F,\P)$ be a complete probability space and $\{B_t:t\ge0\}$ be the standard Brownian
	motion on $\R$. Let $F:\R\times\R\rightarrow\R$ is a continuous function of two variables.
	In the following, we will mainly consider the stochastic process defined by
	\begin{equation}\label{e2.1}
		B_t^F:=\int_0^tF(t,s)\d B_s, \quad t\ge0.
	\end{equation}
	To ensure that $B^F$ is a Gaussian process, we need to add some reasonable conditions for the function $F$:
	\begin{equation}\label{e2.2}
		F\in L^2(\R^2)~~~\text{ and}~~ \int_0^{t\wedge s}F(t,r)F(s,r)\d r\neq0,\quad \forall~t,s>0
	\end{equation}
	and for all $t\ge0$,
	\begin{equation}\label{e2.3}
		\varphi_t(\cdot):=F(t,\cdot)\in L^2(\R).
	\end{equation}
	Under the above conditions \eqref{e2.3} and \eqref{e2.4}, it is easy to show that $B^F$ is a centered Gaussian process with covariance $R_F(t,s):=\EE(B_t^FB_s^F)$ and
	\begin{equation}\label{e2.4}
		R_F(t,s)=\int_0^{t\wedge s}F(t,r)F(s,r)\d r.
	\end{equation}
	Here, we will call $B^F$ as the $F$-Gaussian process.
	
	In particular, when $F(t,s)\equiv1, (t,s)\in \R\times\R$, then $B_t^F$ is the standard Brownian motion on $\R$; when we choose $F(t,s)=K_H(t,s), (t,s)\in [0,\infty)\times[0,\infty)$, where $K_H(t,s)$ is the non-negative square integrable kernel defined on $[0,\infty)\times[0,\infty)$ given by
	\begin{equation}\label{e2.5}
		K_H(t,s)=\begin{cases}
			&c_Hs^{\frac{1}{2}-H}\int_s^t(u-s)^{H-\frac{3}{2}}u^{H-\frac{1}{2}}\d u, \quad \quad\quad \quad \quad\quad\quad \quad \quad \quad\quad\quad ~~ H\in(\frac{1}{2},1),\\
			&1_{[0,t]}(s), \quad \quad \quad\quad\quad\quad \quad \quad\quad\quad \quad \quad \quad\quad\quad \quad \quad \quad\quad\quad\quad \quad ~~  H=\frac{1}{2},\\
			&b_H\left(\left(\frac{t(t-s)}{s}\right)^{H-\frac{1}{2}}-\left(H-\frac{1}{2}\right)s^{\frac{1}{2}-H}\int_s^t(u-s)^{H-\frac{1}{2}}u^{H-\frac{3}{2}}\d u\right), \quad H\in(0,\frac{1}{2}),
		\end{cases}
	\end{equation}
	where $t>s$, and $c_H,b_H$ are independent with respect to $t,s$, $B^{K_H}$ is exactly the fractional Brownian motion with Hurst index $H$ in \cite{MV,BHOZ}.
	
	For each fixed $T>0$, let $\mu^T$ be the law of $B^F_{[0,T]}$,  which is a probability measure on the path space $W_T:=C([0,T];\mathbb{R})$. Usually, $\mu^T$ is also called the $F$-Gaussian Wiener measure.

	Throughout the article, we assume that the conditions \eqref{e2.3} and \eqref{e2.4} hold. In fact, they are also the necessary conditions for $B^F$ to be a Gaussian process.

	\section{Integrals for $F$-Gaussian process}

	In the section, we will introduce stochastic integrals with respect to $F$-Gaussian process $B^F$ by using its Gaussianity. They may be expressed in terms of an integral with respect to the standard Brownian motion. Before going on, we first introduce two operators $K_F^*$ and $(K_F^*)^{-1}$, which are very critical for the quasi-invariant theorem in the next section.

	\subsection{$K_F^*$ and $(K_F^*)^{-1}$ operators}
	
	For any fixed $T>0$, we define by the classical Cameron-Martin space for the standard Brownian motion $B_t$:
	\begin{equation}\label{e3.1}
		\mathbb{H}=\left\{h\in C^{1}([0,T];\R): h(0)=0, \|h\|^2_{\mathbb{H}}:=\int_0^T|h'(s)|^2\d s<\infty\right\},
	\end{equation}
	which is a separable Hilbert space under the inner product
	$$\<h,g\>_\H:=\<h',g'\>_{L^2([0,T])}=\int_0^Th'(s)g'(s)\d s,\quad h,g\in\H.$$
	
	For a fixed function $F\in L^2(\R^2)$ satisfying with  \eqref{e2.3} and \eqref{e2.4}, we will introduce the Cameron-Martin space for the $F$-Gaussian process $B_t^F$. To do that, we need to define an operator $K_F$ on $L^2([0,T];\R)$ given by
	\begin{equation}\label{e3.2}
		K_Fh(t):=\int_0^tF(t,s)h(s)\d s,\quad h\in L^2([0,\infty)).
	\end{equation}
	Let
	\begin{equation}\label{e3.3}
		\H^F:=\left\{h^F:=K_F\dot{h}\Big| h\in \H\right\}.
	\end{equation}
	Then $\H^F$ is a inner product space under the following inner product:
	\begin{equation}\label{e3.4}
		\<h^F,g^F\>_{\H^F}=\int_0^Th'(s)g'(s)\d s
	\end{equation}
	for any $h^F, g^F\in \H^F$ with the forms $h^F:=K_F\dot{h}, ~g^F:=K_F\dot{g}, ~h,g\in\H$. Moreover, we may easily obtain
	$$\<K_F\dot{h},K_F\dot{g}\>_{\H^F}:=\<h,g\>_{\H}, \quad \forall h,g\in \H.$$
	Since
	$$\|K_Fh\|_{L^2}\le\|F\|_{L^2(\R^2)}\|h\|_{L^2},$$
	$K_F$ is a bounded linear operator from $L^2([0,T])$ to $\H^F$. Thus $\H^F$ is also a Hilbert space under the inner product $\<h^F,g^F\>_{\H^F}$.
	
	Let $K_F^{*}$ be the adjoint of $K_F$ in the following sense:
	$$\int_0^T(K_F^{*}g)(s)h(s)\d s=\int_0^Tg(s)\frac{\d(K_Fh)(s)}{\d s},\quad g,h\in L^2([0,T]; \mathbb{R}),$$
	where $g,h\in L^2([0,T])$. By exchanging the order of the double integral, we immediately obtain that
	\begin{equation}\label{e3.5}
		(K_F^{*}g)(s)=F(s,s)g(s)+\int_s^T\frac{\partial F(t,s)}{\partial t}g(t)\d t.
	\end{equation}
	Let $\K_F^{*}$ be the adjoint of $K_F$ in the $L^2$ sense:
	$$\int_0^T(\K_F^{*}g)(t)h(t)\d t=\int_0^Tg(t)K_Fh(t)\d t,\quad g,h\in L^2([0,T]; \mathbb{R}),$$
	where $g,h\in L^2([0,T])$. Moreover, we get easily
	\begin{equation}\label{e3.6}
		\K_F^{*}g(t)=\int_t^TF(s,t)g(s)\d s.
	\end{equation}
	
	\begin{Remark}\label{rem3.1}
		Assume that $F\in L^2([0,T]^2;\mathbb{R})$, then the following hold
		
		$(1)$  $K_F$ and $\K_F^*:L^2([0,T]; \mathbb{R})\rightarrow \mathbb{H}^F$ are bounded linear operators;
		
		
		$(2)$  Further assume that $F\in C^1([0,T];\mathbb{R})$ and $F(s,s)=0, \forall s\in [0,T]$, then $K_F^*:L^2([0,T]; \mathbb{R})\rightarrow\mathbb{H}^F$ is a bounded linear operator, and $\K_F^*=K_F^*\circ\K_{1_{\{s\le t\}}}^*.$
	\end{Remark}

	By the definition of $B^F$, we may take actually the $F$-Gaussian process as the functional of the standard Brownian motion, and combining this with the quasi-invariant property of the Wiener measure, we will derive the quasi-invariant theorem with respect to the $F$-Gaussian Wiener measure. To obtain that, we firstly establish the essential connection between the standard Brownian motion and $F$-Gaussian process. Thus, in addition to the above assumption of $F$ that $B^F$ is a Gaussian process, we need to further assume that:
	\begin{enumerate}
		\item[{\bf(A1)}] For a.s.~$t\in L^2[0,T], F_1(t,s):=\frac{\partial F(t,s)}{\partial t}\in L^2([0,T]^2; \mathbb{R})$.
		\item[{\bf(A2)}] There exists $\Phi\in L^2([0,T]^2; \mathbb{R})\cap C^1((0,T)^2; \mathbb{R})$ such that for any $0\le r\le t\le T$,
		\begin{equation*}
			\xi(r):=\int_r^t\Phi(t,s)F(s,r)\d s~~\text{and}~~\xi,\frac{1}{\xi}\in L^2([0,T]; \mathbb{R}).
		\end{equation*}
		\item[{\bf (A3)}] Assume  $F\in C([0,T]^2;\mathbb{R})$ and there exists $\Theta\in L^2([0,T]^2; \mathbb{R})$ such that for any $0\le t\le r\le T$,
		\begin{equation*}
			\zeta(r):=\Theta(r,t)F(r,r)+\int_t^r\Theta(s,t)\frac{\partial F(r,s)}{\partial r}\d s~~\text{and}~~\zeta,\frac{1}{\zeta}\in L^2([0,T]; \mathbb{R}).
		\end{equation*}
	\end{enumerate}
	
	In the following Theorem \ref{thm3.2}, we will present representations of inverse operators $K_F^{-1}$ and $(K_F^*)^{-1}$ under above certain conditions. This is the crucial point for the proof of the quasi-invariant theorem in the next section.
	
	\begin{thm}\label{thm3.2}
		Suppose {\bf(A1)} holds.
		
		$(1)$ If $F$ satisfies condition {\bf(A2)}, then $K_F$ is invertible and
		\begin{equation}\label{e3.7}
			K_F^{-1}g(t):=\frac{1}{\xi(t)}\frac{\d}{\d t}K_{\Phi}g(t)=\frac{1}{\xi(t)}\Big[\Phi(t,t)g(t)+\int^t_0\frac{\d\Phi(t,s)}{\d t}g(s)\d s\Big], \quad \forall~g\in\H^F.
		\end{equation}
		
		$(2)$ If $F$ satisfies condition {\bf(A3)}, then $K_F^*$ is invertible and
		\begin{equation}\label{e3.8}
			(K_F^*)^{-1}g(t):=-\frac{1}{\zeta(t)}\frac{\d}{\d t}\K_{\Theta}^*g(t)=\frac{1}{\zeta(t)}\Big[\Theta(t,t)g(t)-\int^T_t\frac{\d\Theta(s,t)}{\d t}g(s)\d s\Big], \quad \forall~g\in\H^F.
		\end{equation}
	\end{thm}

	\begin{Remark}$(1)$ Usually, all derivatives $\frac{\d}{\d t}g$ in the above theorem \ref{thm3.2} are in the $L^2$-sense.


	$(2)$  When $F=K_H$, i.e. $B^F$ is the fBm, $K_F^{-1}$ and $(K_F^*)^{-1}$ were introduced by  Decreusefond-\"Ust\"unel in \cite{DU}.

	\end{Remark}

	{\bf Proof of Theorem \ref{thm3.2}}
		$(1)$ $(a)$ Suppose that $h\in L^2([0,T]; \mathbb{R})$, then by \eqref{e2.3} and \eqref{e3.2}, we know that
		$$g(t):=K_Fh\in\H^F.$$
		Now, it suffices to only show that
		$$K_F^{-1} K_Fh(t)=h(t), \quad h\in L^2([0,T]; \mathbb{R})~~\text{and}~~K_F K_F^{-1}g(t)=g(t), \quad g\in \H^F.$$
		Firstly, it is easy to obtain that
		$$K_F^{-1}K_Fh(t)=\frac{1}{\xi(t)}\frac{\d}{\d t}K_{\Phi} K_Fh=h(t).$$
		In fact, using \eqref{e3.2} again,
		$$\aligned  K_{\Phi}K_Fh(t)=&\int_0^t\Phi(t,s)\left(\int_0^sF(s,r)h(r)\d r\right)\d s\\
		=&\int_0^t\int_0^s\Phi(t,s)F(s,r)h(r)\d r\d s\\
		=&\int_0^t\int_r^t\Phi(t,s)F(s,r)h(r)\d s\d r\\
		=&\int_0^t\left(\int_r^t\Phi(t,s)F(s,r)\d s\right)h(r)\d r\\
		=&\int_0^t\xi(r)h(r)\d r,\endaligned$$
		where the last equality is due to {\bf(A2)}. Since $\xi$ is continuous, we obtain
		$$\frac{\d}{\d t}(K_{\Phi}K_Fh)(t)=\xi(t)h(t)---L^2-{sense},$$
		which is equivalent to
		$$h(t)=\frac{1}{\xi(t)}\frac{\d}{\d t}K_{\Phi}K_Fh(t).$$

		$(b)$ Next, we will prove that
		$$K_F K_F^{-1}g(t)=g(t), \quad g\in \H^F.$$
		By the definition of $\H^F$, we know that there exists a $h_1\in L^2([0,T]; \mathbb{R})$ such that $g(t)=K_Fh_1$. Then by (a), we have
		$$K_F K_F^{-1}g(t)=K_F K_F^{-1}K_Fh_1(t)=K_Fh_1(t)=g(t).$$
	
		$(2)$ Proof of $(2)$ is similar to that of $(1)$, for the convenience of the reader, in the following we will give the detailed proof. We need to only show that
		$$\begin{cases}
			&(K_F^*)^{-1} K_F^*h(t)=h(t), \quad h\in L^2([0,T]; \mathbb{R}),\\
			&K_F^* (K_F^*)^{-1}g(t)=g(t), \quad g\in K_F^*(L^2([0,T]; \mathbb{R})).
		\end{cases}$$
		By \eqref{e3.5} and {\bf(A3)}, we have
		$$\aligned \K^T_{\Theta}K_F^*h(t)=&\int_t^T\Theta(s,t)\left(F(s,s)h(s)+\int_s^T\frac{\partial F(r,s)}{\partial r}h(r)\d r\right)\d s\\
		=&\int_t^T\Theta(r,t)F(r,r)h(r)\d r+\int_t^T\int_s^T\Theta(s,t)\frac{\partial F(r,s)}{\partial r}h(r)\d r\d s\\
		=&\int_t^T\Theta(r,t)F(r,r)h(r)\d r+\int_t^T\int_t^r\Theta(s,t)\frac{\partial F(r,s)}{\partial r}h(r)\d s\d r\\
		=&\int_t^T\left(\Theta(r,t)F(r,r)+\int_t^r\Theta(s,t)\frac{\partial F(r,s)}{\partial r}\d s\right)h(r)\d r\\
		=&\int_t^T\zeta(r)h(r)\d r.
		\endaligned$$
		Since $\zeta$ is continuous, we have
		$$\frac{\d }{\d t}\K^T_{\Theta}K_F^*h(t)=-\zeta(t)h(t)---L^2  ~{sense}.$$
		Thus, 
		$$(K_F^*)^{-1} K_F^*h(t)=-\frac{1}{\zeta(t)}\frac{\d}{\d t}\K^T_{\Theta}K_F^*h(t)=h(t).$$
		For any $h\in L^2([0,T]; \mathbb{R})$, let $g=K_F^*h$. By \eqref{e3.5} and  ${\bf (A3)}$, we know that
		$$g(t)=K_F^*h(t)=F(t,t)h(t)+\int_t^T\frac{\partial F(s,t)}{\partial s}h(s)\d s\in L^2([0,T]; \mathbb{R}).$$
		Then by using the above equality, we obtain
		$$ K_F^*(K_F^*)^{-1}g(t)= K_F^*(K_F^*)^{-1}K_F^* h(t)=K_F^*h(t)=g(t).~~~~~~~~~~~~~~~~~~~~~~~~~~~~~\square$$

	In theorem \ref{thm3.2}, we only show that $K_F,K_F^*$ is invertible on the subspace $\H^F$ of $L^2([0,T]; \mathbb{R})$. However, to introduce the integrals with respect to $F$-Gaussian process and the quasi-invariant theorem in the next section, it requires that $K_F^*$ is invertible on the whole $L^2([0,T]; \mathbb{R})$. To do that, we need to add the following stronger condition {\bf(A4)}:

	\begin{enumerate}
		\item[{\bf(A4)}] Assume $F\in C([0,T]^2;\mathbb{R})$ and there exists functions $\Theta\in L^2([0,T]^2;\R)\cap C^1((0,T)^2; \mathbb{R})$ and $0<\zeta_1,\zeta_2\in L^2([0,T]; \mathbb{R})$ with $\frac{1}{\zeta_1},\frac{1}{\zeta_2}\in L^2([0,T]; \mathbb{R})$ such that for any $0\le t\le r\le T$,
		$$\left\{\begin{array}{l}
			\frac{\Theta(r,t)}{\zeta_1(r)}F(r,r)+\int_t^r\frac{\Theta(s,t)}{\zeta_1(s)}\frac{\partial F(r,s)}{\partial r}\d s=\zeta_2(r),\\
			\frac{\Theta(r,t)}{\zeta_2(t)}F(t,t)+\int_t^r\frac{\partial F(s,t)}{\partial s}\frac{\Theta(r,s)}{\zeta_2(s)}\d s=\zeta_1(t),\\
			\frac{\partial\Theta(r,t)}{\partial r}+\frac{\partial\Theta(r,t)}{\partial t}\equiv0.
		\end{array}\right.$$
	\end{enumerate}

	\begin{Remark}\label{rem3.3}
		In the condition {\bf(A4)}, if $\frac{\Theta(s,t)}{\zeta_1(s)}$ and $\zeta_2$ are replaced by $\Theta(s,t)$ and $\zeta$ respectively, then the first equation in the condition {\bf(A4)} implies condition {\bf(A3)}.
	\end{Remark}
	
	\begin{thm}\label{thm3.5}
		Suppose {\bf(A1)} and  {\bf(A4)} hold. Then  we have
		\begin{equation}\label{e3.9}
			\left(K_F^*\right)^{-1}g(t)=-\frac{1}{\zeta_2(t)}\frac{\d}{\d t}\int_t^T\frac{\Theta(s,t)}{\zeta_1(s)}g(s)\d s, \quad g\in L^2([0,T];\mathbb{R}),
		\end{equation}
		where the derivative $\frac{\d}{\d t}$ is in the $L^2$-sense.
	\end{thm}

	\begin{proof}
		It suffices to prove that
		\begin{equation}\label{e3.10}
			\left(K_F^*\right)^{-1}K_F^*g(t)=g(t), \quad K_F^*\left(K_F^*\right)^{-1}g(t)=g(t),\quad g\in L^2([0,T];\mathbb{R}).
		\end{equation}
		According to Theorem \ref{thm3.2}, we know that $\left(K_F^*\right)^{-1}K_F^*g(t)=g(t)$ holds for $g\in L^2([0,T];\mathbb{R})$, and $K_F^*\left(K_F^*\right)^{-1}g(t)=g(t)$ holds for $g\in\H^F$. Thus, we only need to show that
		\begin{equation}\label{e3.11}
			K_F^*\left(K_F^*\right)^{-1}g(t)=g(t), \quad \forall g\in L^2([0,T];\mathbb{R}).
		\end{equation}
		By {\bf(A4)}, we know that $\frac{\Theta(s,t)}{\zeta_1(s)}$ is continous with respect to $s\in [0,T]$. By repeating the above argument in the Theorem \ref{thm3.2}, we may get
		\begin{equation}\label{e3.12}
			\aligned
			\frac{\d}{\d t}\int_t^T\frac{\Theta(s,t)}{\zeta_1(s)}g(s)=&-\frac{\Theta(t,t)}{\zeta_1(t)}g(t)+\int_t^T\frac{\partial\Theta(s,t)}{\partial t}\frac{g(s)}{\zeta_1(s)}\d s\\
			=&-\frac{\Theta(T,t)}{\zeta_1(t)}g(t)+\int_t^T\left(\frac{\partial\Theta(s,t)}{\partial t}\frac{g(s)}{\zeta_1(s)}+\frac{\partial\Theta(s,t)}{\partial s}\frac{g(t)}{\zeta_1(t)}\right)\d s\\
			=&-\frac{\Theta(T,t)}{\zeta_1(t)}g(t)+\int_t^T\frac{\partial\Theta(s,t)}{\partial s}\left(\frac{g(t)}{\zeta_1(t)}-\frac{g(s)}{\zeta_1(s)}\right)\d s----L^2~\text{sense},
			\endaligned
		\end{equation}
		where the second equal is due to the third equation in $(A4)$. Combining this with \eqref{e3.9},
		\begin{equation}\label{e3.13}
			\aligned
			&K_F^*\left(K_F^*\right)^{-1}g(t)\\
			=&-\frac{F(t,t)}{\zeta_2(t)}\left(-\frac{\Theta(T,t)}{\zeta_1(t)}g(t)+\int_t^T\frac{\partial\Theta(s,t)}{\partial s}\left(\frac{g(t)}{\zeta_1(t)}-\frac{g(s)}{\zeta_1(s)}\right)\d s\right)\\
			+&\int_t^T\frac{\partial F(s,t)}{\partial s}\frac{1}{\zeta_2(s)}\left(\frac{\Theta(T,s)}{\zeta_1(s)}g(s)-\int_s^T\frac{\partial\Theta(r,s)}{\partial r}\left(\frac{g(s)}{\zeta_1(s)}-\frac{g(r)}{\zeta_1(r)}\right)\d r\right)\d s
			\endaligned
		\end{equation}

		Next, we will finish the proof by dicussing different case of $g$.

		(a) Firstly, we consider the case of $g(t)=1_{[0,a]}(t)\zeta_1(t)$ for some $a\in (0,T]$. From \eqref{e3.13}, we obtain
		$$\aligned
		K_F^*\left(K_F^*\right)^{-1}g(t)=\frac{\Theta(a,t)}{\zeta_2(t)}F(t,t)+\int_t^a\frac{\partial F(s,t)}{\partial s}\frac{\Theta(a,s)}{\zeta_2(s)}\d s=\zeta_1(t)=g(t), \quad t\leq a,
		\endaligned$$
		where the last second equality is due to the condition ${\bf (A4)}$. When $t>a$, we know that
		$$K_F^*\left(K_F^*\right)^{-1}g(t)=0.$$	
		(b) For general case of $g\in L^2([0,T];\mathbb{R})$. According the conditions $\eta_1,\frac{1}{\eta_1}\in L^2([0,T];\mathbb{R})$, we know that $\{h\eta_1:h\in L^2([0,T];\mathbb{R})\}= L^2([0,T];\mathbb{R})$, thus it only suffices to show that
		$$\aligned
		K_F^*\left(K_F^*\right)^{-1}g(t)=g(t), \quad \forall~g(t)=h(t)\eta_1(t), h\in L^2([0,T];\mathbb{R}).
		\endaligned$$
		Since $K_F^*\left(K_F^*\right)^{-1}$ is a linear operator, by the dominated convergence theorem and $(a)$, it is easy to get the above conclusion. Thus, we finish the proof of this theorem.		
	\end{proof}
	
	\begin{Remark}\label{rem3.6}
		Suppose $F\in C((0,T)^2;\mathbb{R})$ and $F(t,t)\equiv0, t\in [0,T]$, and $\frac{\partial^2F(t,s)}{\partial t^2}$ exists and belongs to $L^2([0,T];\mathbb{R})$, then $\forall h\in L^2([0,T];\mathbb{R})$,
		\begin{equation*}g(t)=K_Fh(t)=\int_0^tF(t,s)h(s)\d s\end{equation*}
		is secondly differentiable and
		$$\begin{cases}
			&g'(t)=\int_0^t\frac{\partial F(t,s)}{\partial t}h(s)\d s, \quad g'(0)=0;\\ &g''(t)=\frac{\partial F(\cdot,t)}{\partial t}(t)h(t)+\int_0^t\frac{\partial^2 F(t,s)}{\partial t^2}h(s)\d s.
		\end{cases}$$
	\end{Remark}
	
	\begin{cor}\label{cor3.7}
		Assume that
		\begin{equation}\label{e3.14}
			F(t,s):=\begin{cases}&\frac{1}{\Gamma(\alpha)}f_1(s)\int_s^t(u-s)^{\alpha-1}f_2(u)\d u, \quad t>s,\\
			&0,~~~~~~~~~~~~~~~~~~~~\quad\quad\quad\quad\quad\quad otherwise,\end{cases}
		\end{equation}
		where $0<\alpha<1$ and $f_1,f_2,\frac{1}{f_1},\frac{1}{f_2}\in L^2([0,T];\R)\cap C^1((0,T);\R)$. Then we have
		\begin{equation}\label{e3.15}
			\begin{cases}
				&K_F^{-1}g(t)=\frac{1}{\Gamma(1-\alpha)f_1(t)}\frac{\d}{\d t}\int_0^t(t-s)^{-\alpha}\frac{g'(s)}{f_2(s)}\d s,~~~~~~~~\quad g\in \H^F\\
				&\left(K_F^*\right)^{-1}g(t)=-\frac{1}{\Gamma(1-\alpha)f_2(t)}\frac{\d}{\d t}\int_t^T(u-t)^{-\alpha}\frac{g(u)}{f_1(u)}\d u,\quad g\in L^2([0,T];\mathbb{R}),
			\end{cases}
		\end{equation}				
		where the derivative $\frac{\d}{\d t}g(t)$ or $g'(t)$ is in $L^2$-sence.
	\end{cor}
	
	\begin{proof}
		$(1)$ Define
		\begin{equation}\label{e3.16}
			F_1(t,s)=\begin{cases}&\frac{1}{\Gamma(\alpha)}(t-s)^{\alpha-1}f_1(s)f_2(t), \quad s\leq t,\\&0,\quad\quad\quad\quad\quad\quad\quad \quad ~~~~~~~ otherwise.\end{cases}
		\end{equation}
		Then $F_1$ satisfies the condition {\bf(A1)}.
		According the definition of $K_F$, for any $h\in L^2([0,T];\mathbb{R})$, we have
		\begin{equation}\label{e3.17}
			\aligned
			g(t):=K_Fh(t)=&\int_0^tF(t,s)h(s)\d s=\int_0^t\left(\int_s^tF_1(u,s)\d u\right)h(s)\d s\\
			=&\int_0^t\left(\int_0^uF_1(u,s)h(s)\d s\right)\d u=\int_0^tK_{F_1}h(u)\d u.
			\endaligned
		\end{equation}
		Let
		$$\Phi_1(t,s)=\frac{(t-s)^{-\alpha}}{f_2(s)},\quad t>s, ~s,~t\in [0,T].$$
		Then
		$$\aligned
		\int_r^t\Phi_1(t,s)F_1(s,r)\d s=&\frac{1}{\Gamma(\alpha)}\int_r^t\frac{(t-s)^{-\alpha}}{f_2(s)}(s-r)^{\alpha-1}f_1(r)f_2(s)\d s\\
		=&\Gamma(1-\alpha)f_1(r):=\xi(r)
		\endaligned$$
		where the last second equal is using the basic formula
		$$\int_0^1(1-s)^{-\alpha}s^{\alpha-1}\d s=\Gamma(\alpha)\Gamma(1-\alpha).$$
		So $F_1$ satisfies the condition {\bf(A2)}. By theorem \ref{thm3.2},
		$$K_{F_1}^{-1}g(t)=\frac{1}{\Gamma(1-\alpha)f_1(t)}\frac{\d}{\d t}\int_0^t(t-s)^{-\alpha}\frac{g(s)}{f_2(s)}\d s.$$
		Since $F$ is continuous, we know that $g$ is differentiable, by \eqref{e3.17} we have
		\begin{equation}\label{e3.18}
			g'(t)=K_{F_1}h(t).
		\end{equation}
		Therefore we have
		$$\aligned
		h(t)=K_{F_1}^{-1}g'(t)=\frac{1}{\Gamma(1-\alpha)f_1(t)}\frac{\d}{\d t}\int_0^t(t-s)^{-\alpha}\frac{g'(s)}{f_2(s)}\d s.
		\endaligned$$
		This finish the proof of the first equation in \eqref{e3.15}.
		
		$(2)$ Next, we will check that $F$ satisfies the condition {\bf(A4)}. According to the definition of $F$, we have
		\begin{equation}\label{e3.19}
			\frac{\partial F(t,s)}{\partial t}=\frac{1}{\Gamma(\alpha)}(t-s)^{\alpha-1}f_1(s)f_2(t).
		\end{equation}
		Let
		$$\begin{cases}&\Theta(t,s)=\frac{1}{\Gamma(1-\alpha)}(t-s)^{-\alpha},\quad s\leq t\\
		&\zeta_1(t)=f_1(t),\quad \quad \quad\quad\quad\quad  ~t\in [0,T]\\
		&\zeta_2(t)=f_2(t),\quad \quad \quad \quad\quad\quad  ~t\in [0,T].\end{cases}$$
		It is easy to check that $\Theta\in L^2([0,T]^2;\R)\cap C^1((0,T)^2)$. By using \eqref{e3.19}, we get
		$$\aligned \int_t^r\frac{\Theta(s,t)}{\zeta_1(s)}\frac{\partial F(r,s)}{\partial r}\d s=&\frac{1}{\Gamma(1-\alpha)\Gamma(\alpha)}\int_t^r(s-t)^{-\alpha}(r-s)^{\alpha-1}f_2(r)\d s=f_2(r)=\zeta_2(r),\\
		\int_t^r\frac{\partial F(s,t)}{\partial s}\frac{\Theta(r,s)}{\zeta_2(s)}\d s=&\frac{1}{\Gamma(1-\alpha)\Gamma(\alpha)}\int_t^r(s-t)^{\alpha-1}(r-s)^{-\alpha}f_1(t)\d s=f_1(t)=\zeta_1(t),\\
		\frac{\partial\Theta(r,t)}{\partial r}+\frac{\partial\Theta(r,t)}{\partial t}=&\frac{-\alpha}{\Gamma(1-\alpha)}\left((r-t)^{-\alpha-1}-(r-t)^{-\alpha-1}\right)=0.\endaligned$$		
		Notice that $F(s,s)\equiv0$ for $0\le s\le T$. Thus $F(t,s)$ satisfies condition {\bf(A4)}. Finally, by theorem \ref{thm3.5}, we get that for all $g\in L^2[0,T]$,
		$$\left(K_F^*\right)^{-1}g(t)=-\frac{1}{f_2(t)\Gamma(1-\alpha)}\frac{\d}{\d t}\int_t^T(u-t)^{-\alpha}\frac{g(u)}{f_1(u)}\d u.$$
	\end{proof}
	
	\begin{cor}\label{cor3.8}
		Assume that $F(t,s)=f(s)$ for some $f\in C([0,T]; \mathbb{R})$ with $\frac{1}{f}\in L^4([0,T];\mathbb{R})$.
	Then we have
		$$\aligned & K_F^{-1}g(t)=\frac{g'(t)}{f(t)}----L^2-sense,\quad\quad\quad g\in \H^F;\\
		&\left(K_F^*\right)^{-1}h(t)=\frac{h(t)}{f(t)},\quad\quad\quad\quad\quad\quad h\in L^2([0,T];\mathbb{R}).\endaligned$$
	\end{cor}
	
	\begin{proof}
		$(1)$ The condition $(A1)$ is obvious. For any $h\in L^2([0,T];\mathbb{R})$, we have
		$$g(t)=K_Fh(t)=\int_0^tf(s)h(s)\d s.$$
		Since $\frac{1}{f}\in L^2([0,T];\mathbb{R})$, we obtain
		$$h(t)=\frac{1}{f(t)}\frac{\d }{\d t}g\big(t\big)----L^2-sense.$$
		This implies
		$$K_F^{-1}K_Fh(t)=K_F^{-1}g(t)=\frac{1}{f(t)}\frac{\d }{\d t}g\big(t\big)=h(t).$$
		The next proof is the same to that of the proof $(b)$ for Theorem \ref{thm3.2}.
	
		$(2)$ Taking $\Theta(t,s)=1, \zeta_1\equiv1,  \zeta_2=f(t)$, then the condition ${\bf (A4)}$ holds. By Theorem \ref{thm3.5}, we finish the proof.
	\end{proof}
	
	\begin{exa}\label{ex3.9}
		$(1)$ In the Corollary \ref{cor3.7}, let
		\begin{equation}\label{e3.20}
			f_1(s)=s^{\frac{1}{2}-H},~~f_2(t)=t^{H-\frac{1}{2}},~~\alpha=H-\frac{1}{2},~~c=c_H\Gamma\Big(H-\frac{1}{2}\Big).
		\end{equation}
		Then $F(t,s)$ is just the kernal $K_H(t,s)$ for the fraction Brownian motion, and
		\begin{equation}\label{e3.21}
			\begin{cases}
				&K_{K_H}^{-1}g(t)=\left(c_H\Gamma(H-\frac{1}{2})\Gamma(\frac{3}{2}-H)\right)^{-1}t^{H-\frac{1}{2}}\frac{\d}{\d t}\int_0^t(t-s)^{\frac{1}{2}-H}s^{\frac{1}{2}-H}g'(s)\d s.\\
				&\left(K_{K_H}^*\right)^{-1}g(t)=-\left(c_H\Gamma(\frac{3}{2}-H)\Gamma(H-\frac{1}{2})\right)^{-1}t^{\frac{1}{2}-H}\frac{\d}{\d t}\int_t^T(u-t)^{\frac{1}{2}-H}u^{H-\frac{1}{2}}g(u)\d u.
			\end{cases}
		\end{equation}
	
		$(2)$ In the Corollary \ref{cor3.7},  let $f_1=f_2\equiv1$, then
		\begin{equation}\label{e3.22}
			F(t,s)=\frac{1}{\Gamma(\alpha)}\int_s^t(u-s)^{\alpha-1}\d u,\quad\frac{\partial F(t,s)}{\partial t}=\frac{1}{\Gamma(\alpha)}(t-s)^{\alpha-1}.
		\end{equation}
		And
		\begin{equation}\label{e3.23}
			\begin{cases}
				&K_F^{-1}g(t)=\frac{1}{\Gamma(1-\alpha)}\left(t^{-\alpha}g'(t)+\int_0^t\left((t-s)^{-\alpha}-t^{-\alpha}\right)g''(s)\d s\right).\\
				&\left(K_F^*\right)^{-1}g(t)=-\frac{1}{\Gamma(1-\alpha)}\left(-(T-t)^{-\alpha}g(t)+\int_t^T\left((u-t)^{-\alpha}-(T-t)^{-\alpha}\right)g'(u)\d u\right).
			\end{cases}
		\end{equation}
		
		$(3)$ In the Corollary \ref{cor3.8}, let $f(s)\equiv C$, where $C\neq0$, then
		\begin{equation}\label{e3.24}
			\begin{cases}
				&K_F^{-1}g(t)=C^{-1}g'(t).\\
				&\left(K_F^*\right)^{-1}g(t)=C^{-1}g(t).
			\end{cases}
		\end{equation}
	\end{exa}
		
	\begin{proof}
		For the functions $F$ of $(1)$ and $(2)$, it is easy to show that all conditions in Corollary \ref{cor3.8} hold. At the same time, for the function $F$ of $(3)$, we also prove easily that all conditions in Corollary \ref{cor3.8} hold. Thus, we obtain all results.
	\end{proof}

	\subsection{Integrals for $F$-Gaussian process}	
	
	In this subsection, we always assume that  {\bf(A1)} and  {\bf(A4)} hold.  Based on the bounded property of the opeator $(K_F^*)^{-1}$, in the following, we will introduce the integrals for $F$-Gaussian process $B_F$.

	For each fixed $T>0$, let
	\begin{equation}\label{e3.25}
		\scr{H}^F:=(K_F^*)^{-1}\big(L^2([0,T];\mathbb{R})\big).
	\end{equation}
	For each $n\in \mathbb{N}$, let $\T_n$ be the set of all partitions $\Delta_n$, which is a partition of the interval $[0,T]$ and satisfy with
	$$\begin{cases}
		&\Delta_n=\{0=t_0^{(n)}<t_1^{(n)}<\dots<t_{l(n)-1}^{(n)}<t_{l(n)}^{(n)}=T\},\\
		&|\Delta_n|=\max\{|t_{i+1}^{(n)}-t_i^{(n)}|:0\le i\le l(n)-1\},\\
		&\lim\limits_{n\rightarrow\infty}|\Delta_n|=0.
	\end{cases}$$
	For any $\psi\in{\scr H}^F$, there exists a function $h\in L^2([0,T];\mathbb{R})$ such that $\psi=(K_F^*)^{-1}h$. For any $\Delta_n\in \T_n$, define
	\begin{equation}\label{e3.26}
		c_i^{(n)}:=\frac{1}{t_{i+1}^{(n)}-t_i^{(n)}}\int_{t_i^{(n)}}^{t_{i+1}^{(n)}}(K_F^*)^{-1}h(s)\d s.
	\end{equation}
	According to conditions of assumptions, we know that $(K_F^*)^{-1}$ is a bounded operator. Thus, we get
	\begin{equation}\label{e3.27}
			\psi_n:=(K_F^*)^{-1}h_n~~ \xrightarrow{L^2}~~(K_F^*)^{-1}h,
	\end{equation}
	where
	$$h_n(s)=\begin{cases} &\frac{1}{t_{i+1}^{(n)}-t_i^{(n)}}\int_{t_i^{(n)}}^{t_{i+1}^{(n)}}h(s)\d s,\quad s\in (t_i^{(n)},t_{i+1}^{(n)}],\\&0,~~~~~~~~~~~~~~~~~~~~~~~~\quad otherwise .\end{cases}$$		
	For each $\Delta_n$, define
	\begin{equation}\label{e3.28}
		\Sigma_n: =\sum_{i=0}^{l(n)-1}c_i^{(n)}\left(B_{t_{i+1}^{(n)}}^F-B_{t_i^{(n)}}^F\right).
	\end{equation}
	
	\begin{lem}\label{lem3.10}
		Under assumptions of {\bf(A1)} and {\bf(A4)}, $\{\Sigma_n\}_{n\geq1}$ is a Cauchy sequence in $L^2(\P)$.
	\end{lem}
	
	\begin{proof}
		By simple caculation, we obtain
		$$\aligned
		&\int_0^t\int_0^s\left(\int_0^{u\wedge v}\frac{\partial F(u,r)}{\partial u}\frac{\partial F(v,r)}{\partial v}\d r\right)\d u\d v\\
		=&\int_0^{t\wedge s}\left(\int_r^t\int_r^s\frac{\partial F(u,r)}{\partial u}\frac{\partial F(v,r)}{\partial v}\d u\d v\right)\d r\\
		=&\int_0^{t\wedge s}\left(\int_r^t\frac{\partial F(v,r)}{\partial v}F(s,r)\d v\right)\d r\\
		=&\int_0^{t\wedge s}F(s,r)F(t,r)\d r.
		\endaligned$$
		Combining this with \eqref{e2.5}, we obtain
		\begin{equation}\label{e3.29}
			\EE(B_t^FB_s^F)=\int_0^{t\wedge s}F(t,r)F(s,r)\d r=\int_0^t\int_0^s\left(\int_0^{u\wedge v}\frac{\partial F(u,r)}{\partial u}\frac{\partial F(v,r)}{\partial v}\d r\right)\d u\d v.
		\end{equation}
		Let		
		\begin{equation}\label{e3.30}
			\phi(u,v)=\int_0^{u\wedge v}\frac{\partial F(u,r)}{\partial u}\frac{\partial F(v,r)}{\partial v}\d r.
		\end{equation}
		Then,
		\begin{equation}\label{e3.31}
			\EE(B_t^FB_s^F)=\int_0^t\int_0^s\phi(u,v)\d u\d v,
		\end{equation}
		thus
		\begin{equation}\label{e3.32}
			\EE(B_{t_2}^F-B_{t_1}^F)(B_{s_2}^F-B_{s_1}^F)=\int_{t_1}^{t_2}\int_{s_1}^{s_2}\phi(u,v)\d u\d v.
		\end{equation}
		By \eqref{e3.28}, for any $m,n\in\N$, we have
		$$\aligned
		\EE(\Sigma_n\Sigma_m)=&\EE\left(\sum_{i=0}^{l(n)-1}c_{i}^{(n)}(B_{t_{i+1}^{(n)}}^F-B_{t_i^{(n)}}^F)\right)\left(\sum_{j=0}^{l(m)-1}c_{j}^{(m)}(B_{t_{j+1}^{(m)}}^F-B_{t_j^{(m)}}^F)\right)\\
		=&\sum_{i=0}^{l(n)-1}\sum_{j=0}^{l(m)-1}c_{i}^{(n)}c_{j}^{(m)}\EE(B_{t_{i+1}^{(n)}}^F-B_{t_i^{(n)}}^F)(B_{t_{j+1}^{(m)}}^F-B_{t_j^{(m)}}^F)\\
		=&\sum_{i=0}^{l(n)-1}\sum_{j=0}^{l(m)-1}c_{i}^{(n)}c_{j}^{(m)}\int_{t_i^{(n)}}^{t_{i+1}^{(n)}}\int_{t_j^{(m)}}^{t_{j+1}^{(m)}}\phi(u,v)\d u\d v\\
		=&\int_0^T\int_0^T\left(\sum_{i=0}^{l(n)-1}c_{i}^{(n)}1_{[t_i^{(n)},t_{i+1}^{(n)}]}(u)\right)\left(\sum_{j=0}^{l(m)-1}c_{j}^{(m)}1_{[t_j^{(m)},t_{j+1}^{(m)}]}(v)\right)\phi(u,v)\d u\d v\\
		=:&\int_0^T\int_0^T\psi_n(u)\psi_m(v)\phi(u,v)\d u\d v.
		\endaligned$$
		Thus
		$$\aligned
		&\EE(\Sigma_n-\Sigma_m)^2=\EE(\Sigma_n^2)-\EE(\Sigma_n\Sigma_m)-\EE(\Sigma_m\Sigma_n)+\EE(\Sigma_m^2)\\
		=&\int_0^T\int_0^T\left(\psi_n(u)\psi_n(v)-\psi_n(u)\psi_m(v)-\psi_m(u)\psi_n(v)+\psi_m(u)\psi_m(v)\right)\phi(u,v)\d u\d v\\
		=&\int_0^T\int_0^T(\psi_n(u)-\psi_m(u))(\psi_n(v)-\psi_m(v))\phi(u,v)\d u\d v\\
		=&\int_0^T\int_0^T(\psi_n(u)-\psi_m(u))(\psi_n(v)-\psi_m(v))\left(\int_0^{u\wedge v}\frac{\partial F(u,r)}{\partial u}\frac{\partial F(v,r)}{\partial v}\d r\right)\d u\d v\\
		=&\int_0^T\left(\int_r^T(\psi_n(u)-\psi_m(u))\frac{\partial F(u,r)}{\partial u}\d u\right)\left(\int_r^T(\psi_n(v)-\psi_m(v))\frac{\partial F(v,r)}{\partial v}\d v\right)\d r\\
		=&\int_0^T\left(K_F^*(\psi_n(r)-\psi_m(r))\right)^2\d r\\
		=&\|K_F^*(\psi_n-\psi_m)\|_{L^2([0,T];\mathbb{R})}^2=\|h_n-h_m\|_{L^2([0,T];\mathbb{R})}^2\rightarrow0.
		\endaligned$$
	\end{proof}
	
	By the above lemma \ref{lem3.10}, we may define the stochastic integral for $F$-Gaussian process $B^F$.
	
	\begin{defn}\label{def3.11}
		Let $\psi\in\scr{H}^F$,
		define the $B^F$-integral with respect to $\psi$
		\begin{equation}\label{e3.33}
			\int_0^T\psi(s)\d B_s^F:=\lim\limits_{\Delta_n\in \T_n, |\Delta_n|\rightarrow0}\sum_{i=0}^{n-1}\left(\frac{1}{t_{i+1}-t_i}\int_{t_i}^{t_{i+1}}\psi(s)\d s\right)\left(B_{t_{i+1}}^F-B_{t_i}^F\right).
		\end{equation}
	\end{defn}	
	
	In the following, we present the connection between stochastic intgeal and $B^F$-integral.
	
	\begin{thm}\label{thm3.12}
		For a fixed $T>0$, we have
		\begin{equation}\label{e3.34}
			\int_0^T\psi(s)\d B_s^F=\int_0^TK_F^*\psi(s)\d B_s,\quad\forall\psi\in\scr{H}^F.
		\end{equation}
	\end{thm}
	
	\begin{proof}
		Since $\psi\in \scr{H}^F$, we may assume that $\psi$ satisfy with \eqref{e3.26} and \eqref{e3.27}.
		By the definition of $B_t^F$, we have
		\begin{equation}\label{e3.35}
			B_t^F=\int_0^TF(t,s)1_{[0,t]}(s)\d B_s=\int_0^TK_F^*1_{[0,t]}(s)\d B_s.
		\end{equation}
		Thus, according to the definition \ref{def3.11}, we have
		$$\aligned
		&\int_0^T\psi(s)\d B_s^F\\
		=&\lim\limits_{n\rightarrow\infty}\sum_{i=0}^{n-1}\left(\frac{1}{t_{i+1}-t_i}\int_{t_i}^{t_{i+1}}\psi(s)\d s\right)(B_{t_{i+1}}^F-B_{t_i}^F)\\
		=&\lim\limits_{n\rightarrow\infty}\sum_{i=0}^{n-1}\left(\frac{1}{t_{i+1}-t_i}\int_{t_i}^{t_{i+1}}\psi(s)\d s\right)\int_0^T\left(K_F^*1_{[0,t_{i+1}]}(s)-K_F^*1_{[0,t_i]}(s)\right)\d B_s\\
		=&\lim\limits_{n\rightarrow\infty}\int_0^TK_F^*\left(\sum_{i=0}^{n-1}\left(\frac{1}{t_{i+1}-t_i}\int_{t_i}^{t_{i+1}}\psi(s)\d s\right)1_{[t_i,t_{i+1}]}(s)\right)\d B_s\\
		=&\lim\limits_{n\rightarrow\infty}\int_0^TK_F^*\left(\psi_n(s)\right)\d B_s\\=&\lim\limits_{n\rightarrow\infty}\int_0^T h_n(s)\d B_s\\
		=&\int_0^TK_F^*\psi(s)\d B_s,
		\endaligned$$
		where the last equal and the second last equal are due to \eqref{e3.27} and \eqref{e3.28}.
	\end{proof}

	\begin{Remark}\label{rem3.13}
		Samko-Kilbas-Marichev \cite{SKM} introduced the Riemann-Liouville fractional integral and derivative operators, which give a simple representation of $K_F$, $K_F^*$ and their inverse operators. In \cite{DU, DU2}, Decreusefond-{\"U}st{\"u}nel represented the Wiener type and divergence type integrals for the fractional Brownian motion, and Alos-Mazet-Nualart \cite{AMN, AMN2, CN} also introduced the divergence type integrals of the fractional Brownian motion. Carmona-Coutin-Montseny \cite{CCM} proved the divergence type integrals for a kind of semi-martingale kernals $K(t,s)$. Nualart \cite{N, N2,N3} obtained the derivative and divergence operators on the Gaussian space and considered the anticipate stochastic differential equations.
	\end{Remark}

	\section{Quasi-invariant Theorem}
	In this section, we will prove that the quasi-invariant theorem for  $F$-Gaussian process $B^F$. Let
	$$R_F:=K_F\circ K_F^*.$$
	Then $R_F$ is the combination of operators $K_F$ and $K_F^*$. Let $\F C_{b}^{OU}$ be the space of all bounded Lipschitz cylinder functions on $W_T$, that is to say, for each $F\in \F C_{b}^{OU}$, there exists some $m\ge1$, $0<t_1<\dots<t_m\le T$ and $f$ is a bounded Lipschitz function on $\R^m$ such that
	$$F=f(\gamma(t_1),\dots,\gamma(t_m)), \quad \gamma\in W_T.$$
	
	\begin{thm}{\bf [Quasi-invariant Theorem]}\label{thm4.1}
		Assume that
		\begin{equation}\label{e4.1}
			F(t,s)=f(t)\hat{F}(t,s),\quad t\geq s\ge0
		\end{equation}	
		for some function $f$ with $f,\frac{1}{f}\in L^2([0,T];\R)$, where $\hat{F}$ satisfies conditions {\bf(A1)}, {\bf(A2)} and {\bf(A4)}. Then for each $h^F\in \H^F$, we have
		$$\int_{\Omega}G\big(B^F(\omega)+h^F\big)\d \P(\omega)=\int_{\Omega}G\big(B^F(\omega)\big)\alpha_{h^F}(\omega)\d\P(\omega)$$
		for any bounded Borel measurable function $G$ on $W_T$, where
		$$\alpha_{h^F}=\exp\left\{\int^T_0(R_{\hat{F}})^{-1}\frac{h^F}{f}\d B^{\hat{F}}-\frac{1}{2}\left\|\frac{h^F}{f}\right\|_{\H^{\hat{F}}}^2\right\}.$$
	\end{thm}
	
	\begin{proof}
		$(a)$ First we assume that $f(t)\equiv1$. In this case $\hat{F}\equiv F$. Since all cylinder functions is dense in $\B_b(\mu^T)$, thus we only need to prove the conclusion holds for every $G\in \F C_{b}^{OU}$. Assume that $G\in \F C_{b}^{OU}$ with
		$$G(\gamma)=g\big(\gamma_{t_1},\cdots,\gamma_{t_m}\big),\quad \gamma\in W_T,$$	
		where $\{t_i\}_{i=1}^m$ is a partiation of $[0,T]$ and $g\in Lip_b(\mathbb{R}^m)$. Then we have
		\begin{equation}\label{e4.2}
			\aligned
			&\int_{\Omega}G\big(B^F+h^F\big)\d\P=\int_{\Omega}G\big(B^{\hat{F}}+h^{\hat{F}}\big)\d\P\\
			&=\int_{\Omega}g\left(\int^{t_1}_0{\hat{F}}(t_1,s)\d B_s+\int^{t_1}_0{\hat{F}}(t_1,s)\dot{h}(s)\d s, \cdots,\int^{t_m}_0{\hat{F}}(t_m,s)\d B_s+\int^{t_m}_0{\hat{F}}(t_m,s)\dot{h}(s)\d s\right)\d \P\\
			&=\int_{\Omega}g\left(\int^{t_1}_0{\hat{F}}(t_1,s)\d (B_s+h_s),\cdots,\int^{t_m}_0{\hat{F}}(t_m,s)\d (B_s+h_s)\right)\d\P\\
			&=\int_{\Omega}g\left(\int^{t_1}_0F(t_1,s)\d (B_s+h_s),\cdots,\int^{t_m}_0F(t_m,s)\d (B_s+h_s)\right)\d\P\\
			&=\int_{\Omega}G\big(B^F\big)\alpha_{h}\d\P,
			\endaligned
		\end{equation}
		where the last equal is due to Girsanov theorem for the Brownian motion and
$$\alpha_{h}=\exp\left\{\int^T_0\dot{h}(t)\d B_t-\frac{\|h\|_\H^2}{2}\right\}.$$
In addition, according to Theorem \ref{thm3.5}, we have
		\begin{equation}\label{e4.3}
			\aligned
			\alpha_{h}&=\exp\left\{\int^T_0K_{\hat{F}}^*(K_{\hat{F}}^*)^{-1}\dot{h}(t)\d B_t-\frac{\|K_{\hat{F}}\dot{h}\|_{\H^{\hat{F}}}^2}{2}\right\}\\
			&=\exp\left\{\int^T_0(K_{\hat{F}}^*)^{-1}\dot{h}(t)\d B^{\hat{F}}_t-\frac{\|K_{\hat{F}}\dot{h}\|_{\H^{\hat{F}}}^2}{2}\right\}\\
			&=\exp\left\{\int^T_0(R_{\hat{F}})^{-1}h^{\hat{F}}\d B^{\hat{F}}-\frac{\|h^{\hat{F}}\|_{\H^{\hat{F}}}^2}{2}\right\}.
			\endaligned
		\end{equation}
		
		$(b)$ For general $F(t,s)=f(t)\hat{F}(t,s)$, let	
		$$\hat{G}(\gamma)=g\big(f(t_1)\gamma_{t_1},\cdots,f(t_m)\gamma_{t_m}\big),\quad \gamma\in W_T,$$		
		where $\{t_i\}_{i=1}^m$ is a partition of $[0,T]$ and $g\in Lip_b(\mathbb{R}^m)$.	
		Then $\hat{G}(\gamma)=\hat{g}(\gamma_{t_1}, \cdots, \gamma_{t_m})$, where
		$$\hat{g}(x_1,\cdots,x_m)=g(f(t_1)\gamma_{t_1}, \cdots, f(t_m)\gamma_{t_m}), ~~~~~x=(x_1,\cdots,x_m)\in \R^m.$$
		Thus $\hat{G}\in \F C^{OU}_b$. Applying $\hat{G}$ into \eqref{e4.2},	
		\begin{equation}\label{e4.4}
			\aligned
			&\int_{\Omega}\hat{G}\big(B^{\hat{F}}+h^{\hat{F}}\big)\d\P=\int_{\Omega}\hat{G}\big(B^{\hat{F}}\big)\alpha_{h}\d\P.
			\endaligned
		\end{equation}		
		Since for each $\omega\in \Omega$,
		\begin{equation}\label{e4.5}
			\aligned
			&\hat{G}\big(B^{\hat{F}}+h^{\hat{F}}\big)\\
			&=g\left(f(t_1)\int_0^{t_1}{\hat{F}}(t_1,s)\d(B_s+h_s),\dots,f(t_m)\int_0^{t_m}{\hat{F}}(t_m,s)\d(B_s+h_s)\right)\\
			&=g\left(B_{t_1}^F+h_{t_1}^F,\dots,B_{t_m}^F+h_{t_1}^F\right)\\
			&=G\big(B^F+h^F\big)
			\endaligned
		\end{equation}	
		and combining \eqref{e4.4} and
		\begin{equation}\label{e4.6}
			\aligned
			&\hat{G}\big(B^{\hat{F}}\big)=G\big(B^F\big),
			\endaligned
		\end{equation}	
	we get to
		\begin{equation}\label{e4.7}
			\aligned
			&\int_{\Omega}G\big(B^F+h^F\big)\d\P=\int_{\Omega}G\big(B^F\big)\alpha_{h}\d\P.
			\endaligned
		\end{equation}	
		Finally, the conclusion is from	the following equality
		\begin{equation}\label{e4.8}
			h_t^F=\int_0^tF(t,s)\dot{h}(s)\d s=f(t)\int_0^t{\hat{F}}(t,s)\dot{h}(s)\d s=f(t)h_t^{\hat{F}}.
		\end{equation}
	\end{proof}

For each $\varepsilon\in \R$ and $h^F\in \H^F$, define the process $B^{F,\varepsilon}:=B^F+\varepsilon h^F$. Let $\mu^F$ and $\mu^{F,\varepsilon}$ are the laws of $B^F$ and $B^{F,\varepsilon}$ respectively.

	\begin{cor}{\bf [Girsanov Theorem]}\label{cor4.2}
		 Assume that $F$ satisfies conditions {\bf(A1)}, {\bf(A2)} and {\bf(A4)}, $B^{F,\varepsilon}$ is an F-Gaussian process under the distribution
		$$\d\mu^{F,\varepsilon}=\alpha_{-\varepsilon h^F}\d\mu^F,$$
		where
		$$\alpha_{h^F}=\exp\left\{\int_0^T(R_H)^{-1}h^F\d B^F-\frac{\|h^F\|_{\H^F}^2}{2}\right\}.$$
	\end{cor}
	
	\begin{proof}
		By theorem \ref{thm4.1}, for all $G\in L^2(\P)$,
		$$\int_{\Omega}G\big(B^{F,\varepsilon}\big)\alpha_{-\varepsilon h^F}\d\P=\int_{\Omega}G\big(B^F\big)\d\P.$$
		Therefore $B^{F,\varepsilon}$ is an $F$-Gaussian process under the distribution $\mu^{F,\varepsilon}$.

		In fact, if we choose two special functions
		$$G_1\big(B^F\big)=B_s^F,\quad G_2\big(B^F\big)=B_s^F B_t^F,$$
		and let $\EE^{F,\varepsilon}$ be the expectation under the distribution $\mu^{F,\varepsilon}$, then
		$$\EE^{F,\varepsilon}(B_s^{F,\varepsilon})=\EE(B_s^F)=0.$$
		And
		$$\EE^{F,\varepsilon}(B_t^{F,\varepsilon}B_s^{F,\varepsilon})=\EE(B_t^FB_s^F)=\int_0^{t\wedge s}F(t,r)F(s,r)\d r.$$
		Moreover, $B^{F,\varepsilon}$ is a continuous Gaussian. Therefore $B^{F,\varepsilon}$ is an $F$-Gaussian process under the distribution $\mu^{F,\varepsilon}$.
	\end{proof}

	\begin{thm}{\bf [Integration by parts formula]}\label{thm4.3}
		Under the same conditions of Theorem \ref{thm4.1}, for each $h^F\in\H^F$ and $G\in\F C_b^{OU}$, we have
		\begin{equation}\label{e4.9}
			\int_{\Omega}D_{h^F}G\big(B^F\big)\d\P=\int_{\Omega}G\big(B^F\big)\beta_{h^F}\d\P,
		\end{equation}
		where
		$$\beta_{h^F}=\int_0^T\big(R_{\hat{F}}\big)^{-1}\frac{h^F(t)}{f(t)}\d B^{\hat{F}}_t.$$
	\end{thm}
	
	\begin{proof}For any $\varepsilon\in[-1,1]$ and $h^F\in \H^F$, we have $\varepsilon h^F\in \H^F$.
		By Theorem \ref{thm4.1}, we have
		\begin{equation}\label{e4.10}
			\int_{\Omega}G\big(B^F+\varepsilon h^F\big)\d\P=\int_{\Omega} G\big(B^F(\omega)\big)\alpha_{\varepsilon h^F}\d\P.
		\end{equation}
		By \eqref{e4.3} in the theorem \ref{thm4.1}, we know that $\{\alpha_{\varepsilon h^F},\varepsilon\in[-1,1]\}$ are $L^2$-integrable martingale. Hence, by the dominated convergnece theorem, and taking the derivative with respect to $\varepsilon$ on the both sides of the above equation \eqref{e4.10}, we have
		\begin{equation}\label{e4.11}
			\aligned&
			\frac{\d\alpha_{\varepsilon h^F}}{\d \varepsilon}\Big|_{\varepsilon=0}=\frac{\d}{\d\varepsilon}\exp\left\{\int^T_0(R_{\hat{F}})^{-1}\frac{\varepsilon h^F}{f}\d B^{\hat{F}}-\frac{1}{2}\left\|\frac{\varepsilon h^F}{f}\right\|_{\H^{\hat{F}}}^2\right\}\Big|_{\varepsilon=0}\\
			=&\exp\left\{\int^T_0(R_{\hat{F}})^{-1}\frac{\varepsilon h^F}{f}\d B^{\hat{F}}-\frac{1}{2}\left\|\frac{\varepsilon h^F}{f}\right\|_{\H^{\hat{F}}}^2\right\}\left(\int_0^T(R_{\hat{F}})^{-1}\frac{h^F}{f}\d B^{\hat{F}}-\varepsilon\left\|\frac{h^F}{f}\right\|_{\H^{\hat{F}}}^2\right)\Big|_{\varepsilon=0}\\
			=&\int^T_0(R_{\hat{F}})^{-1}\frac{h^F(t)}{f(t)}\d B^{\hat{F}}_t.
			\endaligned
		\end{equation}
		We get immediately the conclusion.
	\end{proof}

	\begin{Remark}\label{rem4.4}
		$(1)$ Buckdahn \cite{B2, B3} first proved the Girsanov theorem for a certain transformation $W(\cdot)+\int_{0}^{\cdot}K_s\d s$ of Brownian motion $W$, where $K_s$ is anticipate process. Jien-Ma \cite{JM} proved the anticipate Girsanov theorem for fbm case.
		
		$(2)$ Norros-Valkelia-Virtamo \cite{NVV} gave some applications of the Girsanov theorem like maximum likelihood estimation and so on.
		
		$(3)$Decreusefond-{\"U}st{\"u}nel \cite{DU, DU2} proved the quasi-invariant theorem and Girsanov theorem for the fractional Brownian motion.
	\end{Remark}

	To establish functional inequalities in the next section, in the following we construct three quasi-regular Dirichlet forms on $W_T$, to do that, we first will give serval gradient definitions.

	\begin{defn}\label{def4.5}
		$(1)$ \textbf{Directional derivarive:} For each $h\in L^2([0,T];\R)$ and $G:W_T:\rightarrow\R$ is a function. We say that $D_hG$ is the directional derivative with respect to $h$ if the limition
		\begin{equation}\label{e4.12}
			D_{h}G(\gamma)=\lim\limits_{\varepsilon\rightarrow0}\frac{G(\gamma+\varepsilon h)-G(\gamma)}{\varepsilon}
		\end{equation}
		exists in $L^2(\mu^T)$.
		
		$(2)$ \textbf{Gradient:} Suppose that the directional derivative $D_{h}G$ exists. By Risez representation theorem, there exist three gradients $\nabla^F G$, $\nabla G$ and $D G$ such that
		\begin{equation}\label{e4.13}
			\aligned
			&D_{h}G(\gamma)=\<\nabla^F G(\gamma), h\>_{\H^F},\quad ~~~~\forall~ h\in\H^F,\\
			&D_{h}G(\gamma)=\<\nabla G(\gamma), h\>_{\H},\quad ~~~~~~~~\forall~h\in\H,\\
			&D_{h}G(\gamma)=\<D G(\gamma), h\>_{L^2([0,T];\R)},\quad \forall~ h\in L^2([0,T];\R).
			\endaligned
		\end{equation}
		Here $\nabla^FG(\gamma)$, $\nabla G(\gamma)$ and $DG(\gamma)$ are called the damped gradient, Malliavin gradient and the $L^2$-gradient of $G$ respectively.
	\end{defn}
To construct the following $L^2$-quasi regular Dirichlet form, we still need to introduce a class of new cylindric functions(refer to the argument of the Section 1 in \cite{RMZZ}):
	\begin{equation}\label{e4.14}
		\aligned
		\F C_b^{L^2}=\bigg\{&G(\gamma)=g\bigg(\int_0^Tg_1(s,\gamma_s)\d s,\dots,\int_0^Tg_m(s,\gamma_s)\d s\bigg):\quad\gamma\in W_T, \\
		&g\in Lip_b(\R^m),g_i\in C^{0,1}([0,T]\times\R;\R),1\le i\le m\bigg\}.
		\endaligned
	\end{equation}
It is easy to show that for each $G\in \F C_b^{L^2}$ and $h\in L^2([0,T];\R)$, the directional derivative $D_{h}G$ of $G$ exists.

	\begin{Remark}\label{rem4.6}	
		$(1)$ Suppose that $G\in \F C_b^{OU}$ with
		$$G(\gamma)=g\big(\gamma(t_1),\cdots,\gamma(t_m)\big),\quad\gamma\in W_T$$
		for $g\in Lip_{b}(\R^m),1\le i\le m$. Then for each $h^F\in \H^F$,
		$$\aligned
		D_{h^F}G(\gamma)=&\sum_{i=1}^{m}\partial_ig(\gamma)h^F(t_i)=\sum_{i=1}^{m}\partial_ig(\gamma)\int_0^T1_{[0,t_i]}(s)\d K_Fh'(s)\\
		=&\sum_{i=1}^{m}\partial_ig(\gamma)\<K_F^*1_{[0,t_i]},h'\>_{L^2([0,T];\R)}=\left\<\sum_{i=1}^m\partial_igK_FK_F^*g_i,h^F\right\>_{\H^F}\\
		=&\left\<\sum_{i=1}^{m}\partial_ig(\gamma)R_F(1_{[0,t_i]}),h^F\right\>_{\H^F}=\left\<R_F\left(\sum_{i=1}^m1_{[0,t_i]}(t)\partial_ig\right),h^F\right\>_{\H^F},
		\endaligned$$
		where $\partial_ig$ is the derivative of the $i$-the component of the function $g$.
		Thus by definition \eqref{e4.13},
		\begin{equation}\label{e4.15}
			\nabla^F G(\gamma)(t)=R_F\left(\sum_{i=1}^m1_{[0,t_i]}(t)\partial_ig\right)=R_F(\dot{\overbrace{\nabla F}}),
		\end{equation}
		where
$$\dot{\overbrace{\nabla F}}(t):=\frac{d}{dt}\nabla F(t).$$
		$(2)$ Suppose that $G\in \F C_b^{L^2}$ with
		$$G(\gamma)=g\bigg(\int_0^Tg_1(s,\gamma_s)\d s,\dots,\int_0^Tg_m(s,\gamma_s)\d s\bigg),\quad\gamma\in W_T$$
		for $g\in Lip_b(\R^m),g_i\in C^{0,1}([0,T]\times\R;\R),1\le i\le m$. Then for each $h\in L^2([0,T];\R)$,
		$$D_hG(\gamma)=\sum_{i=1}^{m}\partial_ig\int_0^T\frac{\partial g_i(s,\gamma_s)}{\partial\gamma}h(s)\d s=\left\<\sum_{i=1}^{m}\partial_ig\frac{\partial g_i(s,\gamma_s)}{\partial\gamma},h(s)\right\>_{L^2([0,T];\R)}.$$
		Thus by the definition \eqref{e4.13},
		\begin{equation}\label{e4.16}
			DG(\gamma)(t)=\sum_{i=1}^{m}\partial_ig\frac{\partial g_i(t,\gamma_t)}{\partial\gamma}.
		\end{equation}
	\end{Remark}
	
	\begin{defn}\label{def4.7}
		Define by three quadratic forms:

$(1)$ The damped quadratic form $(\E,\F C^{OU}_b)$
		\begin{equation}\label{e4.17}
			\E(\psi,\psi):=\int_{W_T}\left\|\nabla^F\psi\right\|^2_{\H^F}\d\mu^F,\quad\psi\in\F C^{OU}_b;
		\end{equation}

$(2)$ The O-U quadratic form $(\E_{OU},\F C_b^{OU})$
		\begin{equation}\label{e4.18}
			\E_{OU}(\phi,\phi):=\int_{W_T}\left\|\nabla\phi\right\|^2_{\H}\d\mu^F,\quad\phi\in\F C_b^{OU};
		\end{equation}

$(3)$ The $L^2$ quadratic form $(\E_{L^2},\F C_b^{L^2})$
		\begin{equation}\label{e4.19}
			\E_{L^2}(\varphi,\varphi):=\int_{W_T}\left\|D\varphi\right\|^2_{L^2}\d \mu^F, \quad\varphi\in\F C^{L^2}_b.
		\end{equation}
	\end{defn}

	\begin{thm}\label{thm4.8}
		Assume that $F(t,s)=f(t)\hat{F}(t,s)\quad t\geq s\ge0,$	
		where $\hat{F}$ satisfies conditions {\bf (A1), (A2)} and {\bf (A4)}, and $f$ with $f,\frac{1}{f}\in L^2([0,T])$. Then $(\E,\F C^{OU}_b)$, $(\E_{OU},\F C_b^{OU})$ and
		$(\E_{L^2},\F C_b^{L^2})$ are closable and their closures $(\E,\D(\E))$, $(\E_{OU},\D(\E_{OU}))$ and
		$(\E_{L^2},\D(\E_{L^2}))$ are quasi-regular Dirichlet forms.
	\end{thm}

	\begin{Remark}\label{rem4.9}
		$(1)$ When the based space is a compact Riemannian manifold, Driver-R\"{o}ckner \cite{DR} obtained quasi-regular O-U Dirichlet forms for Wiener measure on the associated path space; L\"{o}bus \cite{L} discussed a class of processes on the path space over a compact Riemannian manifold with unbounded diffusion; Wang-Wu \cite{WW} extended Driver-R\"{o}ckner's result to noncompact Riemannian manifold with the integrable condition for Ricci curvature; Chen-Wu \cite{CW} further shew that conclusions hold for stochastically complete Reimannian manifold without any curvature conditions.

		$(2)$ R\"{o}ckner-Wu-Zhu-Zhu \cite{RMZZ} established quasi-regular $L^2$ Dirichlet forms for Wiener measure on the path space over a compact manifold; Chen-Wu-Zhu-Zhu \cite{CWZZ} extended to a noncompact Riemannian manifold.
	\end{Remark}
	
	{\bf Proof of Theorem \ref{thm4.8}}
		Here we only present the proof of the damped case(for the O-U case and the $L^2$-case, we may prove the result similarly). The proof and notations mainly refer to \cite{CWZZ} (see also \cite{MR}, \cite{RMZZ}, \cite{WW} and \cite{WW2}).
		
		{\bf Closability:} Let $\{\psi_n\}_{n\ge1}\subset\F C_b^{OU}$ be a sequence of cylinder functions such that
		$$\lim\limits_{n\rightarrow\infty}\int_{\Omega}\psi_n^2(B^F(\omega))\d\P=0$$
		and
		$$\lim\limits_{n\rightarrow\infty}\E(\psi_n-\psi_m,\psi_n-\psi_m)=\lim\limits_{n\rightarrow\infty}
		\int_\Omega\|\nabla^F\psi_n-\nabla^F\psi_n\|^2_{\H^F}(B^F(\omega))\d\P=0.$$
		Then $\{\nabla^F\psi_n\circ B^F\}$ is a Cauchy sequence in $L^2(\Omega\rightarrow \H^F;\P)$. Thus there exists a limit $\Psi\in L^2(\Omega\rightarrow \H^F;\P)$ such that
		$$\lim\limits_{n\rightarrow\infty}\nabla^F\psi_n\circ B^F=\Psi------L^2-\text{sense}.$$ We only need to prove that $\Psi=0$. By theorem \ref{thm4.3}, we know that
		$$\int_{\Omega}D_{h^F}G(B^F(\omega))\cdot\d\P=\int_{\Omega}G(B^F(\omega))\cdot\beta_{h^F}(\omega)\d\P$$
		for $G\in\F C_b^{OU}$. Let $\{h_k^F\}_{k\ge1}\in\H^F$ be an orthonormal bases in $\H^F$, then for any $G\in\F C_b^{OU}$, we have
		\begin{equation}\label{e4.20}
			\aligned
			&\int_{\Omega}G(B^F(\omega))D_{h_k^F}\psi_n(B^F(\omega))\d\P\\
			=&\int_{\Omega}D_{h_k^F}(\psi_n\cdot G)(B^F(\omega))\d\P-\int_{\Omega}\psi_n(B^F(\omega))D_{h_k^F}G(B^F(\omega))\d\P\\
			=&\int_{\Omega}D_{h_k^F}\psi_n(B^F(\omega)\cdot G(B^F(\omega))\d\P-\int_{\Omega}\<\nabla^FG(B^F(\omega)),h_k^F\>_{\H^F}\psi_n(B^F(\omega))\d\P\\
			=&\int_{\Omega}\psi_n(B^F(\omega)\cdot G(B^F(\omega))\beta_{h_k^F}(\omega)\d\P-\int_{\Omega}\<\nabla^FG(B^F(\omega)),h_k^F\>_{\H^F}\psi_n(B^F(\omega))\\
			=&\int_{\Omega}\psi_n(B^F(\omega))\left(G(B^F(\omega))\beta_{h_k^F}(\omega)-\<\nabla^FG(B^F(\omega)),h_k^F\>_{\H^F}\right)\d\P.
			\endaligned
		\end{equation}
		Since $G(B^F(\omega))$, $\nabla^FG(B^F(\omega))$ and $\beta_{h_k^F}$ are $L^2$-integratable, thus by the dominated convergence theorem, we get
		\begin{equation}\label{e4.21}
			\aligned
			&\int_{\Omega}\<\Psi(B^F(\omega)),h_k^F\>_{\H^F}G(B^F(\omega))\d\P\\&=\lim\limits_{n\rightarrow\infty}\int_{\Omega}\<\nabla^F\psi_n(B^F(\omega)),h_k^F\>_{\H^F}G(B^F(\omega))\d\P\\
			&=\lim\limits_{n\rightarrow\infty}\int_{\Omega}G(B^F(\omega))D_{h_k^F}\psi_n(B^F(\omega))\d\P\\
			&=\lim\limits_{n\rightarrow\infty}\int_{\Omega}\psi_n(B^F(\omega))\left(G(B^F(\omega))\beta_{h_k^F}(\omega)-\<\nabla^FG(B^F(\omega)),h_k^F\>_{\H^F}\right)\d\P=0\endaligned
		\end{equation}
		for any $G\in\F C_b^{OU}$ and $k\ge1$. This implies that $\Psi\equiv0$.
		
		\textbf{Quasi-regularity:} It is obvious that $\F C_b^{OU}$ is dense in $\D(\E)$ under the $(\E)_1^{1/2}$-norm, and has a countable dense subset separating points on $W_T$. Then, by \cite[Definition IV-3.1]{MR}, to verify the quasi-regularity of $(\scr E, \scr D(\scr E))$ it suffices to find out a sequence of compact sets $\{K_n\}\subset W_T$ such that
		$$\displaystyle\lim_{n\to\infty}\Cap(W_T\backslash K_n)=0,$$
		where $\Cap$ is the capacity induced by $(\E,\D(\E))$.

		$(1)$ We intend to prove that for some constant $C>0$ and any $\gamma\in W_T$, the function
		$$F_\gamma(\sigma):=\displaystyle\sup_{t\in[0,T]}|\gamma_t-\sigma_t|,\quad \sigma\in W_T$$ is in $\D(\E)$ and
		\begin{equation}\label{c2.7}
			\E(F_\gamma,F_\gamma)\leq C.
		\end{equation}
		Let $\{t_k\}_{k=1}^\infty$ be a countable dense subset of $[0,T]$. For each $\gamma\in W_T$, define
		$$F_{\gamma,k}(\sigma):=|\gamma(t_k)-\sigma(t_k)|,\quad F_\gamma^{(n)}:=\max_{1\leq k\leq n}F_{\gamma,k},\quad k,n\in\N,\ \sigma\in W_T.$$
		According to \eqref{e4.15} we derive
		$$\|\nabla^FF_{\gamma,k}\|_{\H^F}^2=\|\nabla F_{\gamma,k}\|_{\H}^2\le T.$$
		Then for each $k\ge1$,
		$$\E(F_{\gamma,k},F_{\gamma,k})=\int_{W_T}\|\nabla^FF_{\gamma,k}\|_{\H^F}^2\d\mu^F\le T.$$
		Thus, for any $n\geq1$,
		$$\E(F_\gamma^{(n)},F_\gamma^{(n)})\le\max_{1\le k\le n}\E(F_{\gamma,k},F_{\gamma,k})\le T,$$
		Since $F_\gamma^{(n)}\uparrow F_\gamma$ as $n\uparrow\infty$ and $F_\gamma\in L^2(W_T,\mu^F)$, it follows \cite[Lemma I-2.12]{MR} that $F_\gamma\in \D(\E)$ and \eqref{c2.7} holds as we select $C\equiv T$.

		(2) Now let $\{\gamma_k:k\geq1\}$ be a countable dense subset of $W_T$. For any $n\geq1$, let
		$$F_n:=\inf_{1\leq k\leq n}F_{\gamma_k}.$$
		Similar to the above argument we obtain from (\ref{c2.7}) that $\{F_n\}\subset\D(\E)$ and $$\E(F_n,F_n)\leq C,\quad n\ge1.$$
		Since the sequence $\{F_n\}_{n\in \mathbb{N}}$ satisfies $F_n\geq0, F_n\geq F_{n+1}, n\in \mathbb{N}$ and $F_1\in L^2(W_T,\mu^F)$, we obtain
		$$\|F_n\|_{L^2(W_T,\mu^F)}\le\|F_1\|_{L^2(W_T,\mu^F)},\quad n\geq1.$$
		Therefore,
		$$\E(F_n,F_n)\leq C+\|F_1\|_{L^2(W_T,\mu^F)}=:C_2,\quad n\geq1.$$
		Since $\{\gamma_i:i\geq1\}\subset W_T$ is a dense subset, we have $F_n\downarrow0$ as $n\uparrow\infty$.
		
		(3) Now, following the line of \cite{DR} (see pages 606 and 607 therein), we need to find out a sequence of closed subsets $\{K_m\subset W_T:m\in\N\}$, such that $F_n\rightarrow0$ uniformly on every $K_m$ and $\lim\limits_{m\rightarrow\infty}\Cap(W_T\backslash K_m)=0$. Since the sequence $\{F_n\}$ is uniformly bounded in (2), there exists a subsequence of $\{F_n\}$ (denoted again by $\{F_n\}$) such that
		$$\E(\bar{F_n},\bar{F_n})\rightarrow0,\quad n\uparrow\infty,$$
		where $\bar{F_n}:=\frac{1}{n}\sum_{i=1}^n F_i, n\in\N$.
		According to \cite[Propositions 3.5]{MR}, there exists a subsequence of $\{\bar{F_n}\}$ (denoted again by $\{\bar{F_n}\}$) and an $\E$-nest $\{K_m\}, m\in\N$, such that on every $K_m$, $\bar{F_n}$ converges uniformly to zero as $n\uparrow\infty$. Since $\{F_n\}$ is a decreasing sequence, $F_n\leq\bar{F_n}$. So $F_n\rightarrow0$ uniformly on every $K_m$ as well. As in \cite{RS}, proof of Proposition 3.1, or in \cite{DR}, proof of Proposition 5, it follows now from the definition of $F_n$ that every $K_m$ is totally bounded. Thus, $\{K_m\}$ is an $\E$-nest consisting of compact sets, which is our desired sequence. $\hfill\square$

	\section{Logarithmic Sobolev inequalities}
	In this section, we will establish Logarithmic Sobolev inequalities with respect to Dirichlet forms $(\E,\D(\E))$, $(\E_{OU},\D(\E_{OU}))$ and $(\E_{L^2},\D(\E_{L^2}))$ respectively.
	First we will present the Clark-Ocone-Haussmann formula for the Gaussian process $B^F$. Let $\F_t^F$ be the $\sigma$-algebra generated by $B^F$, then we know that
	$$\F_t^F:=\sigma\big(B_s^F,s\le t\big)=\sigma\big(B_s,s\le t\big)=\F_t.$$
	
	\begin{thm}{\bf [Clark-Ocone-Haussmann formula]}\label{thm5.1}
		Assume that $F$ satisfies {\bf (A1), (A2)} and {\bf (A4)}, then we have
		\begin{equation}\label{e5.1}
			G\big(B^F\big)=\EE\big[G\big(B^F\big)\big]+\int_0^TH_t^G\d B_t^F,\quad~G\in\D(\E)),
		\end{equation}
		where
		\begin{equation}\label{e5.2}
			H_t^G=(K_F^*)^{-1}\EE\big[K_F^{-1}\nabla^FG\big(B^F\big)(t)|\F_t^F\big].
		\end{equation}
	\end{thm}
	
	\begin{proof}Since $\F C_{b}^{OU}$ is dense in $\D(\E)$, it is enough to prove that \eqref{e5.1} holds for each $G\in \F C_{b}^{OU}$.
		By the classical Clark-Ocone formula of the Brownian motion $B_t$, we know that
		\begin{equation}\label{e5.3}
			G\big(B\big)=\EE\big[G\big(B\big)\big]+\int_0^T\EE\big[DG\big(B\big)|\F_t\big]\d B_t(\omega),\quad\forall~G\in\D.
		\end{equation}
		By the definition of $B^F$, we know that $B^F$ may be looked as the functional of $B$, i.e. $B^F_\cdot(\omega)=\int^\cdot_0F(\cdot,s)\d B_s(\omega):=\Phi\circ B(\omega)$. Then $\Phi\in \D$. Now replace $G$ by $G\circ\Phi$ in \eqref{e5.3}, we have
		\begin{equation}\label{e5.4}
			G\circ\Phi\big(B\big)=\EE\big[(G\circ\Phi)\big(B\big)\big]+\int_0^T\EE\big[D(G\circ\Phi)\big(B\big)|\F_t\big]\d B_t.
		\end{equation}
		Notice that for each $h\in \H$,
		$$\aligned
		\big\<D(G\circ\Phi)(B),\dot{h}\big\>_{L^2[0,T]}&=D_{h}(G\circ\Phi)(B)=D_{K_F\dot{h}}G(\Phi(B))\\&=D_{K_F\dot{h}}G(B^F)
		=\big\<\nabla G(B^F),K_F\dot{h}\big\>_{\H}\\
		&=\int_0^TDG(B^F)\d K_F\dot{h}(t)\\
		&=\big\<K_F^*(DG)(B^F),\dot{h}\big\>_{L^2[0,T]}.
		\endaligned$$
		Since the arbitrariness of $h$, we obtain
		\begin{equation}\label{e5.5}
			D(G\circ\Phi)(B)=K_F^*(DG)(B^F).
		\end{equation}
		By remark \ref{rem4.6}, we have
		$$\aligned
		&G\big(B^F\big)=G\circ\Phi\big(B\big)\\
		=&\EE\big[G\circ\Phi\big(B\big)\big]+\int_0^T\EE\big[D(G\circ\phi)\big(B\big)(t)|\F_t\big]\d B_t\\
		=&\EE\big[G\big(B^F\big)\big]+\int_0^T\EE\big[K_F^*(DG)\big(B^F\big)(t)|\F_t^F\big]\d B_t\\
		=&\EE\big[G\big(B^F\big)\big]+\int_0^T(K_F^*)^{-1}\EE\big[K_F^*(R_F)^{-1}(\nabla^FG)\big(B^F\big)(t)|\F_t^F\big]\d B_t^F\\
		=&\EE\big[G\big(B^F\big)\big]+\int_0^T(K_F^*)^{-1}\EE\big[K_F^{-1}(\nabla^FG)\big(B^F\big)(t)|\F_t^F\big]\d B_t^F.
		\endaligned$$
		Thus we finish the proof.
	\end{proof}

	By theorem \ref{thm5.1}, we may derive the following damped Logarithmic Sobolev inequality.

	\begin{thm}{\bf [Damped Logarithmic Sobolev inequality]}\label{thm5.2}
		Assume that $F$ satisfies {\bf (A1), (A2)} and {\bf (A4)}, we have
		\begin{equation}\label{e5.6}
			\Ent_{\mu^F}(G^2):=\EE_{\mu^F}\left(G^2\log\frac{G^2}{\EE_{\mu^F}(G^2)}\right)\le 2\E(G,G),\quad G\in\D(\E).
		\end{equation}
	\end{thm}
	
	\begin{proof}
		By the Clark-Ocone-Haussmann formula in theorem \ref{thm5.1}, we know that
		\begin{equation}\label{e5.7}
			\aligned
			&G\big(B^F\big)=\EE\big[G\big(B^F\big)\big]+\int_0^TH_t^G\d B_t^F\\
			=&\EE\big[G\big(B^F\big)\big]+\int_0^T\EE\big[K_F^{-1}(\nabla^FG)\big(B^F\big)(t)|\F_t^F\big]\d B_t.
			\endaligned
		\end{equation}
		Without loss of generality, we always assume that $G$ is bounded and $G>0$. Consider the right continuous version of martingale
		\begin{equation}\label{e5.8}
			M_t:=\EE\big[G\big(B^F\big)|\F_t^F\big], ~~\quad 0\le t\le T,
		\end{equation}
		then by \eqref{e5.7} we have
		\begin{equation*}
			M_t=\EE\big[G\big(B^F\big)\big]+\int_0^t\EE\big[K_F^{-1}(\nabla^FG)\big(B^F\big)(t)|\F_t^F\big]\d B_t
		\end{equation*}
		and
		\begin{equation}\label{e5.9}
			\d M_t=\EE\big[K_F^{-1}(\nabla^FG)\big(B^F\big)(t)|\F_t^F\big]\d B_t.
		\end{equation}
		Consider the function $f(x):=x\log x$. Then
		$$f'(x)=\log x+1,\quad f''(x)=x^{-1},\quad x>0.$$
		By the It\^{o}'s formula, we have
		\begin{equation}\label{e5.10}
			\aligned
			f(M_T)-f(M_0)&=\int_0^T(\log(M_t)+1)\d M_t\\
			&+\frac{1}{2}\int_0^T\frac{\left(\EE\big[K_F^{-1}(\nabla^FG)\big(B^F\big)(t)|\F_t^F\big]\right)^2}{M_t}\d t.
			\endaligned
		\end{equation}
		Taking expectation on both side in the above equation with respect to $\EE$, we get
		\begin{equation}\label{e5.11}
			\aligned
			&\EE\big[G(B^F)\log G(B^F)\big]-\EE\big[G(B^F)\big]\log \EE\big[G(B^F)\big]\\
			&=\frac{1}{2}\EE\bigg[\int_0^T\frac{\left(\EE\big[K_F^{-1}(\nabla^FG)\big(B^F\big)(t)|\F_t^F\big]\right)^2}{\EE\big[G\big(B^F\big)|\F_t^F\big]}\d t\bigg].
			\endaligned
		\end{equation}
		By using the Cauchy-Schwarz inequality and $G$ is replaced by $G^2$, we get
		$$\aligned
		&\left(\EE\big[K_F^{-1}(\nabla^FG^2)\big(B^F\big)(t)|\F_t^F\big]\right)^2 \\&
		=\left(\EE\big[2G\big(B^F\big) K_F^{-1}(\nabla^FG)\big(B^F\big)(t)|\F_t^F\big]\right)^2\\&
		\le4\EE\big[G\big(B^F\big)^2|\F_t^F\big]\EE\big[(K_F^{-1}(\nabla^FG)\big(B^F\big)(t))^2|\F_t^F\big]
		\endaligned$$
		Therefore
		$$\aligned
		\Ent_{\mu^F}(G^2)=&\Ent_{\P}\big[G\big(B^F(\omega)\big)^2\big]\\
		\le&2\EE\int_0^T\EE\big[(K_F^{-1}(\nabla^FG)\big(B^F(\omega)\big)(t))^2|\F_t^F\big]\d t\\
		\le&2\EE\int_0^T(K_F^{-1}(\nabla^FG)\big(B^F(\omega)\big)(t))^2\d t\\
		=&2\EE\left\|(\nabla^FG)\big(B^F(\omega)\big)\right\|_{\H^F}^2=2\EE_{\mu^F}\left\|(\nabla^FG)\right\|_{\H^F}^2.
		\endaligned$$
	\end{proof}
	
	From the damped logarithmic Sobolev inequality in Theorem \ref{thm5.2}, we may derive the following Logarithmic Sobolev inequality with respect to Ornstein-Uhlenbeck Dirichlet form $(\E_{OU}, \D(\E_{OU}))$ and $L^2$-Dirichlet form $(\E_{L^2}, \D(\E_{L^2}))$  respectively.
	
	\begin{thm}{\bf [Logarithmic Sobolev inequality(O-U case)]}
		Under the same conditions in Theorem \ref{thm5.2}, for any $\delta>0$ we have
		\begin{equation}\label{e5.12}
			\Ent_{\mu^F}(G^2)\le C\E_{OU}(G,G),\quad \forall~G\in\D(\E_{OU}),
		\end{equation}
	where
		\begin{equation}\label{e5.13}
			C=2(1+\delta)\sup_{t\in[0,T]}F(t,t)+2\frac{1+\delta}{\delta}\left\|\frac{\partial F(t,s)}{\partial t}\right\|_{L^2([0,T]^2;\R)}.
		\end{equation}
	\end{thm}

	\begin{proof}
		By Theorem \ref{thm5.2}, we know that
		\begin{equation}
			\Ent_{\mu^F}(G^2)\le2\E(G,G)=2\int_{W_T}\left\|\nabla^FG\right\|_{\H^F}^2\d \mu^F.
		\end{equation}
		For each $h^F\in\H^F$, we get
		$$\aligned
		&\|\nabla^FG\|_{\H^F}^2=\sup_{\|h^F\|_{\H^F}\le1}\big\<\nabla^FG, h^F\big\>_{\H^F}^2=\sup_{\|h^F\|_{\H^F}\le1}(D_{h^F}G)^2\\
		=&\sup_{\|h^F\|_{\H^F}\le1}\big\<\nabla G,h^F\big\>_{\H}^2\le\sup_{\|h^F\|_{\H^F}\le1}\|h^F\|_{\H}^2\|\nabla G\|_{\H}^2.
		\endaligned$$
		In the following, we need to estimate $\sup_{\|h^F\|_{\H^F}\le1}\|h^F\|_{\H}^2$.
		Noticing that for any $h^F\in\H^F$, we can find $\dot{h}\in L^2([0,T];\R)$ such that $h^F=K_F\dot{h}$. Thus by equation \eqref{e3.4} we obtain
		$$\big\|h^F\big\|_{\H^F}=\big\|K_F\dot{h}\big\|_{\H^F}=\big\|\dot{h}\big\|_{L^2([0,T];\R)}.$$
		For any $\delta>0$ and real numbers $x,y$, we have
		$$(x+y)^2\leq (1+\delta)x^2+\frac{1+\delta}{\delta}y^2.$$
		Applying the above inequality and \eqref{e3.5}, we have
		$$\aligned
		\big\|K_F\dot{h}\big\|^2_{\H}=&\left\|\frac{\d}{\d t}K_F\dot{h}\right\|^2_{L^2([0,T];\R)}=\left\|\frac{\d}{\d t}\int_0^t F(t,s)\dot{h}(s)\d s\right\|^2_{L^2([0,T];\R)} \\
  		\le&\int^T_0\Big[F(t,t)\dot{h}(t)+\int_0^t\frac{\partial F(t,s)}{\partial t}\dot{h}(s)\d s\Big]^2dt
   		\\
  		\le&(1+\delta)\int^T_0\Big[F(t,t)\dot{h}(t)\Big]^2dt+\frac{1+\delta}{\delta}\int^T_0\Big[\int_0^t\frac{\partial F(t,s)}{\partial t}\dot{h}(s)\d s\Big]^2dt\\
		\le&(1+\delta)\sup_{t\in[0,T]}F(t,t)\|\dot{h}\|^2_{L^2([0,T];\R)}+\frac{1+\delta}{\delta}\left\|\frac{\partial F(t,s)}{\partial t}\right\|^2_{L^2([0,T]^2;\R)}\|\dot{h}\|_{L^2([0,T];\R)}
		\endaligned$$
		Therefore
		\begin{equation}\label{e5.15}
			\aligned
			\sup_{\|h^F\|_{\H^F}\le1}\|h^F\|_{\H}^2&=\sup_{\|\dot{h}\|_{L^2([0,T];\R)}\le1}\|K_F\dot{h}\|^2_{\H}\\
			&\le(1+\delta)\sup_{t\in[0,T]}F(t,t)+\frac{1+\delta}{\delta}\left\|\frac{\partial F(t,s)}{\partial t}\right\|_{L^2([0,T]^2;\R)}.
			\endaligned
		\end{equation}
	\end{proof}
	
	\begin{thm}{\bf [Logarithmic Sobolev inequality($L^2$-case)]}\label{thm5.3}
		Under conditions in Theorem \ref{thm5.2}, we get
		\begin{equation}\label{e5.16}
			\Ent_{\mu^F}(G^2)\le C\E_{L^2}(G,G),
		\end{equation}
		for $G\in\D(\E_{L^2})$, where
		\begin{equation}\label{e5.17}
			C=2\|F\|_{L^2([0,T]^2;\R)}^2=2\int_0^T\int_0^TF(t,s)^2\d s\d t.
		\end{equation}
	\end{thm}
	
	\begin{proof}
Since $\F C_b^{L^2}\subset\D(\E_{OU})$, \eqref{e5.6} holds for each $G\in \F C_b^{L^2}$. Next,
		this proof is similar to the proof of Theorem \ref{thm5.3}. For the convenience of readers' reading, here we will provide the complete proof. For any $h\in\H^F$, there exists a  $\dot{h}\in L^2([0,T];\R)$ such that $h^F=K_F\dot{h}$, thus we have
		\begin{equation}\label{e5.18}
			\aligned
			\big\|\nabla^FG\big\|_{\H^F}^2=&\sup_{\big\|h^F\big\|_{\H^F}\le1}\big\<\nabla^FG,h^F\big\>_{\H^F}^2=\sup_{\|h^F\|_{\H^F}\le1}\big|D_{h^F}G\big|^2\\
			=&\sup_{\big\|\dot{h}\big\|_{L^2}\le1}\big|D_{K_F\dot{h}}G\big|^2=\sup_{\|\dot{h}\|_{L^2}^2\le1}\big\<DG,K_F\dot{h}\big\>_{L^2([0,T];\R)}|^2\\
			\le&\big\|DG\big\|_{L^2([0,T];\R)}^2\sup_{\|\dot{h}\|_{L^2}\le1}\big\|K_F\dot{h}\big\|_{L^2([0,T];\R)}^2.
			\endaligned
		\end{equation}
		Now we need to estimate $\sup_{\|\dot{h}\|_2\le1}\|K_F\dot{h}\|_{L^2([0,T];\R)}^2$. In fact,
		$$\aligned
		\|K_F\dot{h}\|_2^2=&\int_0^T\left(\int_0^tF(t,s)\dot{h}(s)\d s\right)^2\d t\\
		\le&\int_0^T\left(\int_0^tF(t,s)^2\d s\int_0^t\dot{h}(s)^2\d s\right)\d t\\
		\le&\int_0^T\left(\int_0^TF(t,s)^2\d s\int_0^T\dot{h}(s)^2\d s\right)\d t\\
		=&\left(\int_0^T\int_0^TF(t,s)^2\d s\d t\right)\left(\int_0^T\dot{h}(s)^2\d s\right)\\
		=&\|F\|_{L^2([0,T]^2;\R)}^2\|\dot{h}\|_{L^2([0,T];\R)}^2.
		\endaligned$$
		This implies conclusion holds.
	\end{proof}

	\begin{Remark}When $B^F$ is the fBm, Fan \cite{F2} obtained the Logarithmic Sobolev inequalities for the O-U Dirichlet form.

	\end{Remark}
	
	\section{Bismut-Elworthy-Li's formula }
	In this section, based on the quasi-invariant theorem \ref{thm4.1}, we will establish Bismut-Elworthy-Li's formula with respect to Gaussian Wiener measure.
	
	For each $x\in \R$, define
	$$B_t^{F,x}=B_t^F+x.$$
Let $\mu$ be the distributions of stochastic processes $B_\cdot^{F,x}$.
	
	\begin{thm}{\bf [Bismut-Elworthy-Li's formula]}\label{thm6.2}
Suppose $F$ satisfies conditions {\bf (A1), (A2)} and {\bf (A4)}. Then for each $T>0$ and $g\circ B_T^{F,x}\in L^1(\P)$,
		\begin{equation}\label{e6.1}
			\frac{\d}{\d x}\EE\left(g\big(B_T^{F,x}\big)\right)=\frac{1}{T_0}\EE\left(g\big(B_T^{F,x}\big)\left(\int_0^T(R_F)^{-1}\left(\int_0^tF(t,s)\d s\right)\d B_t^{F,x}\right)\right).
		\end{equation}
		where
		$$T_0:=\int_0^TF(T,s)\d s.$$
	\end{thm}
	
	\begin{proof}
		According to the condition {\bf (A2)}, we  know that $T_0>0$. For each $\varepsilon\geq0$ and $t\geq0$, define
		$$h^{F,\varepsilon}_t:=K_F\Big(\frac{\varepsilon}{T_0}\Big)(t)=\int^t_0F(t,s)\frac{\varepsilon}{T_0}ds=\frac{\varepsilon}{T_0}\int^t_0F(t,s)ds.$$
		Then $h^{F,\varepsilon}\in \H^F$. Denote
		$$B_t^{F,x,\varepsilon}=B_t^{F,x}-h_t^{F,\varepsilon},\quad t\geq0.$$
		Let $\mu^{\varepsilon}$ be the distributions of stochastic processes $B_\cdot^{F,x,\varepsilon}$. According to \ref{cor4.2}, we have
		\begin{equation}\label{e6.2}\aligned
			\frac{\d\mu^{\varepsilon}}{\d\mu}&=\exp\left\{-\int_0^{T_0}(R_F)^{-1}h^{F,\varepsilon}\d B_t^F-\frac{1}{2}\|h^{F,\varepsilon}\|_{\H^F}^2\right\}\\
			&=\exp\left\{-\int_0^{T_0}(R_F)^{-1}h^{F,\varepsilon}\d B_t^F-\frac{\varepsilon^2}{2{T_0}}\right\}.\endaligned
		\end{equation}
		In additional, for any $T_1>0$, the law of  $B_{[0,T_1]}^{F,x+\varepsilon,\varepsilon}$ under the probability measure $\mu^{\varepsilon}$ is the same as $B_{[0,T_1]}^{F,x+\varepsilon}$ under $\mu$, i.e.
		\begin{equation}\label{e6.3}
			\int_\Omega G\big(B_{[0,T_1]}^{F,x+\varepsilon,\varepsilon}\big)\frac{\d\mu^{\varepsilon}}{\d\mu}\d \P=\int_\Omega G\big(B_{[0,T_1]}^{F,x+\varepsilon}\big)\d \P,\quad \forall ~G\in L^1(\mu).
		\end{equation}
		Since
		$$\aligned B_T^{F,x+\varepsilon,\varepsilon}&=B_T^{F,x+\varepsilon}-h_T^{F,\varepsilon}=x+\varepsilon+B_T^{F}-h_T^{F,\varepsilon}\\
		&=x+\varepsilon+B_T^{F}-\frac{\varepsilon}{T_0}\int^T_0F(T,s)ds\\
		&=x+\varepsilon+B_T^{F}-\varepsilon\\
		&=B_T^{F,x},\endaligned$$
		Combing all the above arguments, in particular, taking $G(\gamma)=g(\gamma_T)$, we obtain that
		$$\int_\Omega g\big(B_T^{F,x}\big)\frac{\d\mu^{\varepsilon}}{\d\mu}\d \P=\int_\Omega g\big(B_T^{F,x+\varepsilon,\varepsilon}\big)\frac{\d\mu^{\varepsilon}}{\d\mu}\d \P=\int_\Omega g\big(B_T^{F,x+\varepsilon}\big)\d \P$$
		Thus,
		\begin{equation}\label{e6.4}
			\int_{\Omega}g\big(B_T^{F,x+\varepsilon}\big)\d\P=\int_{\Omega}g\big(B_T^{F,x}\big)\exp\left\{-\int_0^{T_0}(R_F)^{-1}h^{F,\varepsilon}\d B_t^F-\frac{\varepsilon^2}{2{T_0}}\right\}\d\P.
		\end{equation}
		Taking the derivative with respect to $\varepsilon$ on both sides of the above equation, and let $\varepsilon=0$, we immediately get
		\begin{equation}\label{e6.5}
			\frac{\d}{\d x}\EE\left(g(B_T^{F,x})\right)=\frac{1}{T_0}\EE\left(g(B_T^{F,x})\left(\int_0^T(R_F)^{-1}\left(\int_0^tF(t,s)\d s\right)\d B_t^{F,x}\right)\right).
		\end{equation}
	\end{proof}
	
	\begin{Remark} 
When $B^F$ is the fractional Brownian motion, Theorem \ref{thm6.2} cover Theorem 3.1 in \cite{F}.
	\end{Remark}
	
	\section{Martingales with respect to F-Gaussian Process}
	In general, the $F$-Gaussian process $(B_t^F)_{t\ge0}$ is not a martingale. A natural question: Does  there exists a measurable function $h(t,s)$ such that the stochastic integral
	$$M(t):=\int_0^th(t,s)\d B_s^F$$
	is a martingale? The following theorem \ref{thm7.1} gives the positive answer.
	
	\begin{thm}\label{thm7.1}
		Assume that $F$ satisfies conditions {\bf (A1)} and {\bf (A4)}, then
		\begin{equation}\label{e7.1}
			M(t)=\int_0^t\frac{\Theta(t,s)}{\zeta_2(s)}\d B_s^F
		\end{equation}
		is a martingale with respect to the filtration generated by the $F$-Gaussian process, where $\Theta$ and $\zeta_2$ are coincide with that in the condition {\bf (A4)}.
	\end{thm}
	
	\begin{proof}
		According to \eqref{e7.1} and \eqref{e3.34} we get
		$$\aligned
		M(t)=&\int_0^TK_F^*\left(\frac{\Theta(t,s)}{\zeta_2(s)}1_{[0,t]}(s)\right)\d B_s\\
		=&\int_0^T\left(F(s,s)\frac{\Theta(t,s)}{\zeta_2(s)}1_{[0,t]}(s)+\int_s^T\frac{\partial F(u,s)}{\partial u}\frac{\Theta(t,u)}{\zeta_2(u)}1_{[0,t]}(u)\d u\right)\d B_s\\
		=&\int_0^t\left(F(s,s)\frac{\Theta(t,s)}{\zeta_2(s)}+\int_s^t\frac{\partial F(u,s)}{\partial u}\frac{\Theta(t,u)}{\zeta_2(u)}\d u\right)\d B_s.
		\endaligned$$
		By the cndition {\bf (A4)}, we know
		\begin{equation}\label{e7.2}
			F(s,s)\frac{\Theta(t,s)}{\zeta_2(s)}+\int_s^t\frac{\partial F(u,s)}{\partial u}\frac{\Theta(t,u)}{\zeta_2(u)}\d u:=\zeta_1(s)
		\end{equation}
		only depend on the variable $s$ and is $L^2$-integrable. This implies that
		$$M(t):=\int_0^t\zeta_1(s)\d B_s$$
		 is a martingale.
	\end{proof}

	\begin{exa}\label{ex7.2}
		$(1)$ When $F(t,s)=K_H(t,s)$.
		Then
		$$M(t)=\int_0^t(t-s)^{\frac{1}{2}-H}s^{\frac{1}{2}-H}\d B_s^{H}$$
		is a martingale.
		
		$(2)$ When
		$$F(t,s)=\frac{1}{\Gamma(\alpha)}f_1(s)\int_s^t(u-s)^{\alpha-1}f_2(u)\d u,$$
		Then
		$$M(t)=\int_0^t(t-s)^{-\alpha}f_2(s)^{-1}\d B_s^F$$
		is a martingale.
		
		$(3)$ When
		$$F(t,s)=f(s),$$
		then for any $\zeta(s)\in L^2([0,T])$,
		$$M(t)=\int_0^t\zeta(s)\d B_s^F$$
		is a martingale. In fact, $F$ is independent with respect to $t$, hence
		$$K_F^*\zeta(s)=f(s)\zeta(s)$$
		is independent with respect to $t$.
		\begin{proof} We only prove $(2)$. By the proof Corollary \ref{cor3.7}, we know that $F$ satisfies {\bf (A1)} and {\bf (A4)}.  we obtain easily that
			$$\Theta(t,s)=(t-s)^{-\alpha},\quad\zeta_2(s)=f_2(s),\quad\zeta_1(s)=\Gamma(1-\alpha)f_1(s).$$
 		Hence conclusions are from Theorem \ref{thm7.1}.
		\end{proof}
	\end{exa}

	\begin{Remark}\label{rem7.3}
		When $F(t,s)=K_H(t,s)$, Norros-Valkeila-Virtamo \cite{NVV} proved that $M(t)$ is a martingale by using the covariance of $M(s),M(t)$, and Hu-Nualart-Song \cite{HDJ} used $M(t)$ to prove some properties of fractional martingales.
	\end{Remark}

	\section{Applications to Mathematical Finance}
	In this section, we will apply the techiques developed in previous sections to Mathematical finance. It is well known that in financial derivatives pricing theory,
	we use geometric Brownian motion to model stock processes. Let $S_{t}$ be the stock process and satisfies the following stochastic differential equation
	\begin{equation}\label{e8.1}
		\d S_{t}=\mu S_{t}\d t+\sigma S_{t}\d B_{t},\quad S_0=s\in \R.
	\end{equation}
	where $B_{t}$ is the standard Brownian motion and $\mu \in \mathbb{R}$ is the drift coefficient and $\sigma >0$ is the diffusion
	coefficient also called volatility, both are constants. By It\^{o} formula, we know that it 
	has a solution
	\begin{equation}\label{e8.2}
		S_{t}=S_{0}\exp\bigg\{\Big(\mu-\frac{1}{2}\sigma^2\Big)t^2+\sigma B_{t}\bigg\}.
	\end{equation}
	Furthermore we consider the European call option with payoff function
	\begin{equation}\label{e8.3}
		P(S_{T})=\max(S_{T}-K,0)
	\end{equation}
	where $S_{T}$ is the stochastic process $S_{t}$ at time $T$ and constant $K\in \mathbb{R}$ is the strike price. By the Black-Scholes theory, to price the derivative, we have to use the risk neutral measure and  work with risk free interest rate $r$. Then we can price this derivative by the following formula
	\begin{equation}\label{e8.4}
		V(0,S_{0},K)=S_{0}N(d_{1})e^{-rT}-K N(d_{2}),
	\end{equation}
	where
	$$N(x)=\frac{1}{\sqrt{2\pi}}\int^x_{-\infty}e^{-\frac{t^2}{2}d t}$$
	and	$$\aligned
	d_{1}&=\frac{\log\frac{S_{0}}{K}+rT+\frac{1}{2}\sigma^2T}{\sigma\sqrt{T}}\\
	d_{2}&=\frac{\log\frac{S_{0}}{K}+rT-\frac{1}{2}\sigma^2T}{\sigma\sqrt{T}}.
	\endaligned$$
	However, derivative pricing under this assumption does not match the observations from  market data. In option market, we find that option with strike price $K$ has an implied volatility $\sigma$ depends on $K$ where  in our model, the volatility $\sigma$ does not depend on the strike price $K$.  This dependence can be characterized as a curve called volatility skew or volatility smile. Therefore new models are being called to explain this phenomina. One of the models is the stochastic volatility model. In the stochastic volatility model, instead of using one source of uncertainty $B_{t}^{(1)}$ we introduce another source of uncertainty $B_{t}^{(2)}$
	$$\aligned
	\d S_{t}&=rS_{t}+\sqrt{v_{t}}S_{t}\d B_{t}^{(1)},\\
	\d v_{t}&=a(v_{t})dt+b(v_{t})\d B_{t}^{(2)},
	\endaligned$$
	where $a,b$ are two functions to characterize the variance process of $v_t$, $B_{t}^{(1)},B_{t}^{(2)}$ are two Brownian motions with certain correlation.
	
	Besides the traditional Brownian motion, recently people starts to use Frational Brownian motion as the driver for $v_{t}$, namely rough volatility models. In the rough volatility model, they use fractional Brownian motion to characterize variance
	\begin{equation}\label{e8.5}
		v_{t}=v_0\exp\left(B^{H}_{t}\right),
	\end{equation}
	where $B^{H}$ is the fractional Brownian motion with respect to Hurst index of $H$, $v_{t}$ starts from $v_{0}$.
	
	Using the technique developed in this paper, we can further extend the above model
	\begin{equation}\label{e8.6}
		v_{t}=v_0\exp\left(B^{F}_{t}\right),
	\end{equation}
	where
	\begin{equation}\label{e8.7}
		B^{F}_{t}=\int_{0}^{t}F(t,s)\d B_{s}.
	\end{equation}
	Here $F(t,s)$ is one of the permissible kernels defined in previous sections. We now list some of the properties of the variance processes.

	\begin{thm}\label{thm8.1}
		The variance process defined in \eqref{e8.6} has mean
		\begin{equation}\label{e8.8}
			E(v_{t})=\exp\left\{\frac{1}{2}\int_{0}^{t} F^2(t,s)\d s\right\}
		\end{equation}
		and has Malliavin Derivative
		\begin{equation}\label{e8.9}
			D_{s}v_{t}=v_{t}F(t,s).
		\end{equation}
		And further more we have
		\begin{equation}\label{e8.10}
			D_{s}\left(\frac{1}{v_{t}}\right)=-\frac{1}{v_{t}}F(t,s)
		\end{equation}
	\end{thm}
	
	\begin{proof}
		For each fixed $t>0$, consider the stochastic process
		$$M_u:=\int_0^uF(t,s)\d B_s, \quad 0\le u\le t.$$
		Then $\{M_u:0\le u\le t\}$ is a martingale under the filtraction generated by standard Brownian process $\{B_u:0\le u\le t\}$. Then
		$$\exp\left\{\int_0^uF(t,s)\d B_s-\frac{1}{2}\int_0^uF^2(t,s)\d s\right\},\quad 0\le u\le t$$
		is the exponential martingale of $\{M_u:0\le u\le t\}$. Therefore for any $0\le u\le t$,
		$$\EE\left(\exp\left\{\int_0^uF(t,s)\d B_s-\frac{1}{2}\int_0^uF^2(t,s)\d s\right\}\right)=1.$$
		In particular, when $t=u$, that is
		$$\EE(v_t)=\exp\left\{\frac{1}{2}\int_0^tF^2(t,s)\d s\right\}.$$
		For the Malliavin Derivative, for any $h\in\H$, we have
		$$\aligned
		D_h(v_t)=&\frac{\d}{\d\varepsilon}\bigg|_{\varepsilon=0}\exp\left(\int_0^tF(t,s)\d(B_s+\varepsilon h_s)\right)\\
		=&\lim\limits_{\varepsilon\rightarrow0}\frac{1}{\varepsilon}\left(\exp\Big(B_t^F+\varepsilon\int_0^tF(t,s)\dot{h}(s)\d s\Big)-\exp(B_t^F)\right)\\
		=&\exp(B_t^F)\int_0^tF(t,s)\dot{h}(s)\d s\\
		=&\int_0^tv_tF(t,s)\dot{h}(s)\d s=\int_0^tD_s(v_t)\dot{h}(s)\d s.
		\endaligned$$
		That is
		$$D_s(v_t)=v_tF(t,s).$$
		Similarly, we have
		$$\aligned
		D_h\bigg(\frac{1}{v_t}\bigg)=&\frac{\d}{\d\varepsilon}\exp\left(-\int_0^tF(t,s)\d(B_s+h_s)\right)
		\\=&-\int_0^t\frac{1}{v_t}F(t,s)\dot{h}(s)\d s=\int_0^tD_s\Big(\frac{1}{v_t}\Big)\dot{h}(s)\d s.
		\endaligned$$
		That is
		$$D_s\Big(\frac{1}{v_t}\Big)=-\frac{1}{v_t}F(t,s).$$
		Thus we have proved equations \eqref{e8.8}, \eqref{e8.9} and \eqref{e8.10}.
	\end{proof}

	Consider now a derivative on variance whose payoff is $f(v_{t})$. In finance the function $f$ is usually a continuous function with reasonable smoothness. For example in the case of standard call option, the function $f$ takes the following form
	\begin{equation*}
  		f(x) = \max(c - x, 0), c\in \mathbb{R}
	\end{equation*}
	We now want to price its value at time $0$. Let it's value be called
	\begin{equation}\label{e8.11}
		P(v_{0})=\EE\left(f(v_{t})\right),
	\end{equation}
	here we want to emphasis the dependence of the function $P$ on the $v_{0}$. When $f\in C^1_b$, its derivative with respect to initial value $v_{0}$ can be calculated directly as
	\begin{equation}\label{e8.12}
		\frac{\d P}{\d v_{0}}=\EE\left(f'(v_{t})\frac{v_{t}}{v_{0}}\right).
	\end{equation}
	By using integration by parts formula and Theorem \ref{thm8.1} we have
	\begin{equation}\label{e8.13}
		Df(v_{t})=f'(v_{t})Dv_{t}=f'(v_{t})v_{t}F(t,s),
	\end{equation}
	hence we have
	\begin{equation}\label{e8.14}
		f'(v_{t})v_{t}=Df(v_{t})\frac{1}{F(t,s)}.
	\end{equation}
	Usian the integration by parts formula again, we get
	\begin{align*}
		\EE\left(f'(v_{t})v_{t}\right)&=\EE(Df(v_{t})\frac{1}{F(t,s)})\\
		&=\EE\left(f(v_{t})\int_{0}^{t}\frac{1}{F(t,s)}\d B_{s}\right)\\
		&=\EE\left(f(v_{t})B^{1/F}_{t}\right)
	\end{align*}

	For general case, when $f$ is not a $C^1_b$ function, we can use the previous Bismut's formula and can still get a nice expression of the expectation which is friendly to numerical computation.
	\begin{thm}\label{thm8.2}
		For each fixed $t>0$, we have
		\begin{equation}\label{e8.15}
		\frac{\d P}{\d v_{0}}=\frac{1}{v_{0}}\EE\big[f(v_{t})B^{1/F}_{t}\big]
		\end{equation}
		for any function $f$ satisfying with  $f\circ \exp (B^{F,x}_{t})\in L^1(\P).$
	\end{thm}

	\begin{proof} 
		According to \eqref{e8.6} and Theorem \ref{thm6.2}
		\begin{align*}\label{e8.15}
			\frac{\d P}{\d v_{0}}&=\frac{\d }{\d v_{0}}\EE\Big(f\Big(v_0\exp\left(B^{F}_{t}\right)\Big)\Big)=\frac{\d }{\d v_{0}}\EE\Big[ f\Big(e^{\log v_{0} +B_{t}^{F}}\Big)\Big] \\
             & = \frac{1}{v_{0}} \frac{\d}{\d x} \EE\Big[  f(e^{x +B_{t}^{F}} )\Big]~~\Big(\text{here}~ x=\log v_0\Big)\\
             & =\frac{1}{v_{0}} \EE\Big[  f(v_{t}) B^{1/F}_{t}\Big]\\
             & =\frac{1}{v_{0}}\EE\big[f(v_{t})B^{1/F}_{t}\big].
		\end{align*}
	\end{proof}
	This proof can reduce the smoothness requirements for function $f$ which is essential for a lot of application in finance. For example, the plain vanilla call, put, digital or barrier options all have payoff functions fall into this category.
	\newpage

\end{document}